\documentclass[11pt,a4paper, reqno]{amsart}
\usepackage{amssymb,mathrsfs,graphicx}
\usepackage{epsfig,subfigure}
\usepackage{ifthen}
\usepackage{pdfsync, bm}
\usepackage{colortbl}
\definecolor{black}{rgb}{0.0, 0.0, 0.0}
\definecolor{red}{rgb}{1.0, 0.5, 0.5}

\usepackage{ifthen} 
\provideboolean{shownotes} 
\setboolean{shownotes}{true} 
\newcommand{\margnote}[1]{
\ifthenelse{\boolean{shownotes}}%
{\marginpar{\raggedright\tiny\texttt{#1}}}%
{}%
}
\newcommand{\hole}[1]{
\ifthenelse{\boolean{shownotes}}%
{\begin{center} \fbox{ \rule {.25cm}{0cm} \rule[-.1cm]{0cm}{.4cm}
\parbox{.85\textwidth}{\begin{center} \texttt{#1}\end{center}} \rule
{.25cm}{0cm}}\end{center}} {} }

\title[Kinetic Flocking--Navier-Stokes model]{On the analysis of a coupled kinetic-fluid model with local alignment forces}

\author[Carrillo]{Jos\'e A. Carrillo}
\address[Jos\'e A. Carrillo]{\newline Department of Mathematics \newline
Imperial College London, London SW7 2AZ, United Kingdom}
\email{carrillo@imperial.ac.uk}

\author[Choi]{Young-Pil Choi}
\address[Young-Pil Choi]{\newline Department of Mathematics \newline
Imperial College London, London SW7 2AZ, United Kingdom}
\email{young-pil.choi@imperial.ac.uk}

\author[Karper]{Trygve K. Karper}
\address[Trygve K. Karper]{\newline Department of Mathematical Sciences \newline
Norwegian University of Science and Technology, Trondheim, N-7491, Norway}
\email{karper@gmail.com}

\numberwithin{equation}{section}

\newtheorem{theorem}{Theorem}[section]
\newtheorem{lemma}{Lemma}[section]
\newtheorem{corollary}{Corollary}[section]
\newtheorem{proposition}{Proposition}[section]
\newtheorem{remark}{Remark}[section]
\newtheorem{definition}{Definition}[section]

\newcommand{\bbr}{\mathbb R}

\newcommand{\bbt}{\mathbb T}

\newcommand{\mt}{\mathcal{T}}

\def\charf {\mbox{{\text 1}\kern-.30em {\text l}}}

\newcommand{\eps}{\varepsilon}
\newcommand{\Grad}{\nabla}
\newcommand{\vr}{\varrho}
\newcommand{\Div}{\Grad \cdot }

\newcommand{\T}{\mathbb{T}}
\newcommand{\R}{\mathbb{R}}
\renewcommand{\S}{\mathcal{S}}

\renewcommand{\H}{\mathcal{H}}
\newcommand{\weak}{\rightharpoonup}
\newcommand{\Set}[1]{\left\{#1\right\}}

\newcommand{\vc}[1]{{\bm #1}}

\begin{document}
\allowdisplaybreaks

\date{\today}



\begin{abstract}
This paper studies global existence, hydrodynamic limit,
and large-time behavior of weak solutions to a kinetic flocking model
coupled to the incompressible Navier-Stokes equations.
The model describes the motion of particles immersed in a Navier-Stokes fluid
interacting through local alignment. We first prove the existence of
weak solutions using energy and $L^p$ estimates together with the
velocity averaging lemma. We also rigorously establish a
hydrodynamic limit corresponding to strong noise and local alignment.
In this limit, the dynamics can be totally described by a coupled compressible
Euler - incompressible Navier-Stokes system. The proof is via
relative entropy techniques. Finally, we show a conditional result on the large-time behavior
of classical solutions. Specifically, if the mass-density satisfies a uniform in time integrability estimate, then particles align with the fluid velocity exponentially fast without any further assumption on the viscosity of the fluid.
\end{abstract}

\maketitle \centerline{\date}

\setcounter{tocdepth}{1}
\tableofcontents

%
%
%
%
\section{Introduction}
In the animal kingdom, one can find several species where the action of  individuals leads to large coherent structures and where there are no external forces or ``leader" guiding the interaction. Perhaps the most famous examples are flocks of birds, schools of fish, or insect swarms. However, similar phenomena in self-organization are also relevant for bacteria, in robotic engineering, and in material science. The past decade has witnessed a massive growth in the attempts to develop mathematical models capturing these types of phenomena. These models are usually based on incorporating different mechanisms of interaction between the individuals such as local repulsion, long-range attraction, and alignment. These Individual Based Models lead to macroscopic descriptions by means of mean-field limit scalings, see \cite{CFTV} for a review. These continuum descriptions can be written as kinetic equations in which there is a mechanism of interaction in the velocity or orientation vector. A very simple idea implementing the consensus mechanism in velocity was introduced by Cucker and Smale in \cite{CS} and improved recently in \cite{MT}. These models take into account nonlocal interactions of the particles by averaging in velocity space. Here, we will focus on a much stronger local averaging of the velocity vector and the effect of a fluid in the tendency to consensus. We will explain the relation to these classical models of alignment below.

The model under consideration governs the motion of particles immersed in a Navier-Stokes fluid interacting through local alignment. By local alignment, we mean that each particle actively tries to align its velocity to that
of its closest neighbors. The particles and fluid are coupled through linear friction. If we let $f = f(x,\xi,t)$ be the one-particle distribution function at a spatial periodic domain $x \in \bbt^3, \xi \in \bbr^3$ at time $t$, and $u =u(x,t)$ be the bulk velocity of fluid, then our model reads
\begin{align}\label{V-INS}
    \partial_t f + \xi \cdot \nabla_x f &= \alpha \nabla_{\xi} \cdot \big[ (\xi - u)f\big] + \beta \nabla_{\xi} \cdot \big[ (\xi - u_f)f\big]  + \sigma \Delta_{\xi} f \nonumber\\
    \partial_t u + u \cdot \nabla_x u + \nabla_x p &= \mu \Delta_x u - \alpha \rho_f ( u - u_f) \\
    \Div u &= 0 \nonumber
\end{align}
subject to initial data
\begin{equation}\label{ini-V-INS}
f(x,\xi,0) = f_0(x,\xi), \quad u(x,0) = u_0(x),
\end{equation}
where $\alpha,\beta,\sigma > 0$ are constants, and $\rho_f$ and $u_f$ denote the average local density and velocity, respectively
\begin{equation}\label{def-locals}
\rho_f :=\int_{\bbr^3} f d\xi, \quad  \rho_f u_f := \int_{\bbr^3} \xi f d\xi.
\end{equation}
The model \eqref{V-INS} contains as particular cases two previously studied models in the literature. If $\beta = 0$, the model reduces to the fluid-particle model studied in \cite{GJV,GJV2}, see also \cite{BDGM,CG,Ham,MV,MV2}. They analyzed the existence of weak solutions and their hydrodynamic limit. On the other hand, if $\alpha = 0$, \eqref{V-INS} decouples and becomes the kinetic flocking model studied in \cite{KMT, KMT2, KMT3}. This latter series of papers establish existence of weak solutions and hydrodynamic limit, but leaves out the question of large-time behavior. 

In this paper, we shall be concerned with the case $\alpha$, $\beta > 0$. This introduces new difficulties compared to the previous studies, requiring non trivial arguments to overcome them. To prove existence of weak solutions to \eqref{V-INS}, the main challenges are posed by the product $fu_f$ and the lack of regularity on $u$. In the first case, weak compactness of $fu_f$ is not trivial as there does not seem to be any available regularity in a spatial domain. Moreover, $u_f$ is only defined on regions with $\vr_f > 0$ and hence does not belong to any $L^p$-space. In this paper, we will obtain the needed compactness from the velocity averaging lemma together with some technical arguments. This  part of the proof will be similar to the existence proof in
\cite{KMT} for \eqref{V-INS} with $\alpha = 0$. However, the coupling with the Navier-Stokes equations introduces new problems that are not straightforward to handle.

Since the equation \eqref{V-INS} is posed in $2d +1$ dimensions, finding an approximate solution is  computationally expensive. For this reason, it is of interest to identify regimes
where the complexity of the equations reduces. In this paper, we shall rigorously identify one such regime corresponding to strong noise and local alignment. That is, the case where $\beta
\sim \sigma \sim \eps^{-1}$, where $\eps$ is a small number. We will establish that in this
case $f$ is close to a thermodynamical equilibrium $f \sim c_0\vr_f e^{-|u_f-\xi|^2/2}$
and that the dynamics can be well approximated by a compressible Euler equation
for $(\vr_f, u_f)$ coupled to the incompressible Navier-Stokes equations for $u$ (See Section \ref{subs2-2} for clarity). We will achieve this result by establishing a relative entropy inequality. Though this type of inequality was originally devised in \cite{Daf} to prove weak-strong uniqueness results, it has also been successfully applied to hydrodynamic limits for kinetic equations \cite{GLT,Lan,Yau}. The perhaps most relevant study is \cite{KMT2}, where \eqref{V-INS} with $\alpha = 0$ is studied. However, with $\beta > 0$, deriving a relative entropy bound is more involved and requires completely new arguments that does not
have a kin in the literature.

For the estimates of large-time behaviour of solutions, when $\beta = 0$, i.e., no local alignment force, the particle-fluid equations \eqref{V-INS} reduces to the Vlasov-Navier-Stokes-Fokker-Planck equations. For this system, classical solutions near Maxwellians converging asymptotically to them were constructed in \cite{GHMZ}. More recently, the incompressible Euler-Fokker-Planck equations ($\beta = 0$ and $\mu = 0$) were treated in \cite{CDM} showing the existence of a unique classical solution near Maxwellians converging to them. On the other hand, without the diffusive term ($\sigma = 0$), the particle-fluid system has no trivial equilibria, and as a consequence the previous arguments used in \cite{CDM,GHMZ} for the estimates of large-time behaviour can not be applied. The large-time behaviour of the Vlasov-Navier-Stokes equations, to our knowledge, have only been studied in \cite{BCHK2,BCHK3}. By replacing the Cucker-Smale alignment force in \cite{BCHK2} by the local alignment one, we will show the emergence of alignment between fluid and particles as time evolves.

Let us now give some explanation for the term local alignment and how this pertains to the Cucker-Smale flocking model.
In the previous decade, Cucker \& Smale \cite{CS} introduced a Newtonian-type flocking model
 using $\ell^2$-based arguments:
\begin{equation}\label{level-par}
\frac{dx_i}{dt} = \xi_i, \quad \frac{d\xi_i}{dt} = \sum_{j=1}^N\psi^{cs}_{ij}(\xi_j - \xi_i), \quad t > 0, \quad i \in \{1,\dots,N\},
\end{equation}
where $x_i(t) \in \R^d$ and $\xi_i(t) \in \R^d$ are the position and velocity of $i$-th particles at time $t$, respectively and  were $\psi^{cs}_{ij}$ is a communication weight between particles defined by
\begin{equation}\label{cs-comm}
\psi^{cs}_{ij}:= \frac1N\psi^{cs}(|x_i-x_j|), \quad i,j \in \{1,\dots,N\}.
\end{equation}
Subsequently, this flocking model and its invariants have been extensively studied in a vast number of papers
such as \cite{CCR,CFRT,HL,HT} to mention a few.
However, more recently Motsch and Tadmor pointed out several
deficiencies with the Cucker-Smale model, and suggested a new model which take into account not only distance between particles but also their relative distance \cite{MT}. More precisely, they considered a nonsymmetric communication weight normalized with a local average density:
\begin{equation}\label{mt-comm}
\psi^{mt}_{ij}:= \frac{\psi^{cs}_{ij}}{\sum_{k=1}^N\psi^{cs}_{ik}}
\end{equation}
As a result, the Motsch-Tadmor model does not involve any explicit dependence on the number of particles. Since $\psi_{ij}^{mt}$ is nonsymmetric, they introduce a new tools based on the notion of active sets to estimate the flocking behavior of particles.

On the other hand, when the number of particles goes to infinity, $N \to \infty$, one can formally derive a mesoscopic description for system \eqref{level-par}-\eqref{mt-comm} with density function $f=f(x,\xi,t)$ which is a solution to the Vlasov-type equation:
\begin{displaymath}\label{kinetic-mt}
\left\{ \begin{array}{ll}
& \partial_t f + \xi\cdot \nabla f + \nabla_{\xi} \cdot (F[f]f) = 0,\\[2mm]
& F[f](x,\xi) := \frac{\int_{\R^d \times \R^d}  \psi^{cs}(|x-y|)(\xi_* - \xi)f(y,\xi_*)dyd\xi_*}{\int_{\R^d} \psi^{cs}(|x-y|)\rho_f(y) dy},\\[2mm]
& f_0(x,\xi) := f(x,\xi,0).
\end{array}\right.
\end{displaymath}
Now, notice that $F[f]$ can be rewritten as
\[
F[f]= \tilde u_f - \xi \quad \mbox{where} \quad \tilde u_f := \frac{\int_{\R^d \times \R^d}  \psi^{cs}(|x-y|)\xi_*f(y,\xi_*)dyd\xi_*}{\int_{\R^d} \psi^{cs}(|x-y|)\rho_f(y) dy}.
\]
and hence that this equation is a non-local version of $\eqref{V-INS}_1$. However, we can localize the previous derivation by assuming that the communication rate is very concentrated around the closest neighbors of a given particle, i.e, that $\psi^{cs}(x)$ is close to a Dirac Delta at the origin. Under this localization of the alignment, it is reasonable to expect $\eqref{V-INS}_1$ as the $N \rightarrow \infty$ limit of the Motsch-Tadmor model. Some formal indications on its validity are provided in \cite{KMT3}.

In the next section, we state our main results. Then, in Section 3, we provide a priori energy and $L^p$ estimates. Section 4 is devoted to  the proof of global existence of weak solutions using Schauder's fixed point theorem and velocity averaging lemma.  In Section 5, we rigorously investigate the convergence of weak solutions to the system \eqref{V-INS} when the local alignment and diffusive forces are sufficiently strong.
In Section 6, we show  a priori estimates for long-time behavior of solutions.

\textbf{Notation.-} We provide several simplified notations that are used throughout the paper. For a function $f(x,\xi)$ ($(x,\xi)\in \T^3 \times \R^3$), we denote by $\|f\|_{L^p}$ the usual $L^p(\bbt^3 \times \bbr^3)$-norm, and if $u$ is a function of $x \in \bbt^3$, we denote by $\|u\|_{L^p}$ the usual $L^p(\bbt^3)$-norm, otherwise specified. We also drop $x$-dependence of differential operators $\partial_{x_i}$, $\nabla_x$, and $\Delta_x$, i.e., $\partial_if := \partial_{x_i} f $, $\nabla f := \nabla_x f$ and $\Delta f := \Delta_x f$.

\section{Main results}
In this section, we state the three main results
of this paper. Our first result concerns the existence of
global weak solutions to \eqref{V-INS}. In the second
result, we rigorously study a hydrodynamic limit
of \eqref{V-INS} corresponding to strong noise
and strong local alignment.
 Our final result is an estimate
on the large-time behavior of solutions to \eqref{V-INS} with $\sigma = 0$.
The latter result assumes that the solutions are sufficiently integrable and the particle density is uniformly bounded in time.

\subsection{Existence of weak solutions}
Let us define
\[
\mathcal{H} := \{ w\in L^2(\bbt^3)~|~ \nabla_x \cdot w
 = 0\},\quad \mathcal{V} :=\{ w\in H^1(\bbt^3)~|~\nabla_x \cdot w = 0\}.
\]
and denote by $\mathcal{V}^{\prime}$ the dual space of $\mathcal{V}$.

Existence will be proved using the following notion of weak solutions.

\begin{definition}\label{def-weak} Suppose the initial data $(f_0, u_0)$ satisfy
\begin{equation}\label{init_cond1}
f_0 \in (L^1_+ \cap L^{\infty})(\bbt^3 \times \bbr^3), \quad |\xi|^2 f _0 \in L^{1}(\bbt^3 \times \bbr^3), \quad u_0 \in \mathcal{H}.
\end{equation}

For a given $T \in (0, \infty)$, we say that the pair $(f, u)$ is a weak solution of \eqref{V-INS}-\eqref{ini-V-INS}
provided the following conditions are satisfied:
\begin{enumerate}
\item
$ f\in L^{\infty}(0,T; (L^1_+ \cap L^{\infty})(\bbt^3\times\bbr^3)), \quad
|\xi|^2 f\in L^{\infty}(0,T;L^1(\bbt^3 \times\bbr^3)).$
\item
$ u\in L^{\infty}(0,T;\mathcal{H})\cap L^2(0,T;\mathcal{V})
\cap \mathcal{C}^0([0,T],\mathcal{V}^{\prime}).$
\item For all $\phi \in \mathcal{C}^1(\bbt^3 \times\bbr^3 \times [0, T))$
with $\phi(\cdot,\cdot,T) = 0$,
\begin{align*}
&-\int_0^T \int_{\bbt^3 \times \bbr^3}
f\left(\partial_t\phi +\xi\cdot \nabla_x \phi~d\xi\right) dxds \nonumber \\
&\quad \qquad - \int_0^T \int_{\bbt^3 \times \bbr^3}\big( \alpha(u - \xi)f + \beta (u_f - \xi)f - \sigma \nabla_{\xi}f \big) \cdot \nabla_{\xi}\phi~ d\xi dx ds
 \cr
&\quad \qquad  = \int_{\bbt^3\times\bbr^3}
 f_0 \phi (\cdot,\cdot,0)\,d\xi dx.
\end{align*}
\item For all $ \psi\in [\mathcal{C}^1(\bbt^3 \times [0, T])]^3$, and $\nabla_x\cdot\psi =0,~\mbox{for a.e.}~t$,
\begin{align*} 
&\int_{\bbt^3}
 u(t)\cdot \psi(t) dx
 +\int_0^t \int_{\bbt^3}
\left(- u\cdot\partial_t\psi -
(u\cdot\nabla_x)\psi \cdot u + \mu
\nabla_x u : \nabla_x \psi \right)dx ds \notag\\
& \hspace{1cm} =-\int_0^t \int_{\bbt^3}
(u-u_f)\cdot\psi \rho_f\,dx ds + \int_{\bbt^3} u_0 \cdot
\psi(0,\cdot)\,dx.
\end{align*}
\end{enumerate}
\end{definition}

Our existence result is given by the following theorem.
\begin{theorem}\label{weak-thm} Suppose the initial data $(f_0, u_0)$ satisfies \eqref{init_cond1}.
Then for any $T>0$ there exists at least one weak solution $(f, u)$ to \eqref{V-INS}-\eqref{ini-V-INS} on the time-interval $(0, T)$.
\end{theorem}

The proof of Theorem \ref{weak-thm} is the topic of Section \ref{sec4}.

\subsection{Hydrodynamic limit}\label{subs2-2}
In our second result, we study the regime where
the noise and local alignment are relatively strong compared
to the other terms. That is, for $\eps$ small, we
consider the system
\begin{align}\label{seq-V-INS}
\begin{aligned}
&\partial_t f^\eps + \xi \cdot \nabla_x f^\eps + \nabla_{\xi} \cdot \left[ (u^\eps - \xi)f^\eps\right] = \frac{1}{\eps} \nabla_\xi \cdot \left[ \nabla_\xi f^\eps - (u_{f^\eps} - \xi)f^\eps\right],\cr
&\partial_t u^\eps + u^\eps \cdot \nabla_x u^\eps + \nabla_x p^\eps - \mu \Delta_x u^\eps = - \int_{\R^3} (u^\eps - \xi) f^\eps d\xi,\cr
&\nabla_x \cdot u^\eps = 0.
\end{aligned}
\end{align}
Now, observe that the right-hand side can be written
\[
\nabla_\xi \cdot \left[ \nabla_\xi f^\eps - (u_{f^\eps} - \xi)f^\eps \right] = \nabla_\xi \cdot \left( M^\eps \nabla_\xi \left( \frac{f^\eps}{M^\eps}\right) \right),
\]
where we have introduced the Maxwellian
\[
M^\eps(x,\xi,t) := \frac{1}{(2\pi)^{3/2}}e^{-\frac{|\xi - u_{f^\eps}(x,t)|^2}{2}}.
\]
Consequently, if we have that $u_{f^\eps} \to u_f$ and $u^\eps \to u$, then
we expect that $f^\eps$ converges to the thermodynamical  equilibrium
\[
f^\eps \to M_{\rho_f,u_f}(x,\xi,t):= \frac{\rho_f(x,t)}{(2\pi)^{3/2}}e^{-\frac{|\xi - u_f(x,t)|^2}{2}} \quad \mbox{as} \quad \eps \to 0.
\]

In this case, it can be readily seen that $\rho_f,u_f$, and $u$ evolves according to the fluid-particle model
\begin{align}\label{hydro-eqn}
\begin{aligned}
&\partial_t \rho_f + \nabla_x \cdot \rho_f u_f = 0,\cr
&\partial_t (\rho_f u_f) + \nabla_x \cdot (\rho_f u_f \otimes u_f) + \nabla_x \rho_f = \rho_f(u-u_f),\cr
&\partial_t u + u \cdot \nabla_x u + \nabla_x p - \mu \Delta_x u = -\rho_f(u-u_f),\cr
&\nabla_x \cdot u =0,
\end{aligned}
\end{align}
subject to
\begin{equation}\label{ini-hydro-eqn}
(\rho_f(x,0),u_f(x,0),u(x,0)) = (\rho_{f_0},u_{f_0},u_0), \quad x \in \T^3.
\end{equation}

In our second result fact, we prove that weak solutions of \eqref{seq-V-INS}
are close to a strong unique solution of \eqref{hydro-eqn}.
Hence, if $\eps$ is sufficiently small, \eqref{hydro-eqn}
provides a good approximation of \eqref{seq-V-INS}. 

\begin{theorem}\label{lim-thm} Assume that there exists a unique strong solution $(\rho_f,u_f,u)$ to the system \eqref{hydro-eqn}-\eqref{ini-hydro-eqn} in the interval $[0,T^*]$. Furthermore suppose that $(f_0,u_0)$ satisfies \eqref{init_cond1}, and $f_0$ is given by
\[
f_0(x,\xi) = \frac{\rho_{f_0}(x)}{(2\pi)^\frac32}e^{-\frac{|\xi - u_{f_0}(x)|^2}{2}}.
\]
Then, for any sequences of weak solutions $(f^\eps,u^\eps)$ to the system \eqref{seq-V-INS}, we have
\begin{equation*}
    \begin{split}
        \sup_{ 0 \leq t \leq T^*} \left(\| u_{f^\eps} - u_{f}\|_{L^2}^2
        +\|\rho_{f^\eps} - \rho_f\|_{L^2}^2
        +\| u^{\eps} - u\|_{L^2}^2\right)
        \leq C \sqrt{\eps}.
    \end{split}
\end{equation*}
As a consequence, as $\eps \to 0$,
\begin{align*}
\begin{aligned}
&f^\eps \to \frac{\rho_f}{(2\pi)^\frac32} e^{-\frac{|\xi - u_f|^2}{2}} \quad \mbox{in } L^1_{loc}(0,T^*;L^1(\T^3 \times \R^3)),\cr
&\rho_{f^\eps}u_{f^\eps} \to \rho_f u_f \quad \mbox{in } L^1_{loc}(0,T^*;L^1(\T^3)),\cr
&\rho_{f^\eps}|u_{f^\eps}|^2 \to \rho_f |u_f|^2 \quad \mbox{in } L^1_{loc}(0,T^*;L^1(\T^3)),\cr
&u^\eps \to u \quad \mbox{in } L^1_{loc}(0,T^*;L^2(\T^3)),\cr
\end{aligned}
\end{align*}

\end{theorem}

\subsection{Large-time behavior}
Our third result is a large-time behavior estimate for our kinetic model. To state this result, we introduce several energy-fluctuation functions:
\begin{align*}
    {\mathcal E}_P(t) &:= \frac{1}{2}\int_{\bbt^3 \times \bbr^3}|\xi - u_f|^2 f dx d\xi \\
    \mathcal{E}_U(t) &:= \frac{1}{2}\int_{\bbt^3 \times\bbt^3} |u_f(x) - u_f(y)|^2 \rho_f(x) \rho_f(y) dx dy \\
    \mathcal{E}_F(t) &:= \frac{1}{2}\int_{\bbt^3} |u - u_c(t) |^{2} dx \\
    \mathcal{E}_I(t) &:= \frac{1}{2}|u_c(t)-\xi_c(t)|^2,
\end{align*}
where $u_c$ and $\xi_c$ are the mean bulk velocity of the fluid and the averaged particle velocity:
\[
u_c := \int_{\bbt^3} u~dx \quad \mbox{and} \quad \xi_c:=\int_{\bbt^3 \times \bbr^3} \xi f ~d\xi dx.
\]
We finally set a total energy function $\mathcal{E}$:
\[
\mathcal{E}(t) := 2\mathcal{E}_P(t) + \mathcal{E}_U(t) + 2\mathcal{E}_F(t) + \mathcal{E}_I(t).
\]
For this analysis, without loss of generality, we assume $\alpha = \beta = 1$.
\begin{theorem}\label{a-priori-est-lt} Let $(f,u)$ be global in time classical solutions to the system \eqref{V-INS}-\eqref{ini-V-INS} with $\sigma = 0$ satisfying
\[
\mathcal{E}(0)<\infty,\qquad \lim_{|\xi| \to \infty} |\xi|^2 f (x,\xi,t) = 0, \qquad (x,t) \in \bbt^3 \times [0,\infty).
\]
Assume that $\|\rho_f\|_{L^{\infty}(0,\infty;L^{3/2}(\bbt^3))} < \infty$, then the total energy fluctuation function $\mathcal{E}(t)$ satisfies
\[
\frac{d}{dt}\mathcal{E}(t) \leq -C\mathcal{E}(t), \quad \mbox{for} \quad t \in [0,\infty),
\]
where $C$ is a positive constant depending on $\mu, \rho_f$.
\end{theorem}
\begin{remark}
Since the total momentum $u_c(t) + \xi_c(t)$ is conserved, we find
\[
\frac12\mathcal{E}_I(t) = \left|u_c(t) - \frac12(\xi_c(0) + u_c(0))\right|^2 = \left|\xi_c(t) - \frac12(\xi_c(0) + u_c(0))\right|^2.
\]
Thus this deduces the emergence of exponential alignment between particles and fluid, and they asymptotically converge to half of the initial total momentum. Notice that the previous theorem makes no assumption on the viscosity of the fluid.
\end{remark}

%
%
%
%
\section{Preliminary material}
The purpose of this section is to derive
  {\it a priori} energy and
$L^p$ estimates for the system \eqref{V-INS}. We will also provide two technical lemmata
that will be frequently applied  in the subsequent analysis.
In this process, we shall use the following notations for the $k$-th local and global momentums
\begin{equation*}\label{notations}
m_k(f)(x,t) = \int_{\bbr^3} |\xi|^k f d\xi, \quad M_{k}(f)(t) := \int_{\bbt^3 \times \bbr^3} |\xi|^k f(x,\xi) dx d\xi,
\end{equation*}
where $k=0,1 \ldots$. We also observe that
\[
\rho_f = m_0(f)(x,t), \quad |\rho_f u_f | \leq m_1(f)(x,t), \quad \mbox{and} \quad M_k(f)(t) = \int_{\bbt^3} m_k(f) dx.
\]
\subsection{A priori energy and $L^p$ estimate}

The following proposition provides an
energy estimate.
\begin{proposition}\label{energy-prop}
Let $(f,u)$ be any fast decaying at infinity smooth solutions to the system \eqref{V-INS}. Then, the following properties hold
\begin{align*}
\begin{aligned}
(i)~ & \frac{d}{dt}\int_{\bbt^3 \times \bbr^3} f d\xi dx = 0,\\
 (ii)~& \frac{d}{dt}\left( \int_{\bbt^3} u dx + \int_{\bbt^3 \times \bbr^3} \xi f d\xi dx \right) = 0\cr
(iii)~ & \frac{1}{2} \frac{d}{dt} \left( \int_{\bbt^3 \times \bbr^3} |\xi|^2 f d\xi dx + \int_{\bbt^3} |u|^2 dx \right) + \mu \int_{\bbt^3} |\nabla u|^2 dx \cr
& = - \alpha \int_{\bbt^3 \times \R^3} |u - \xi|^2 f d\xi dx  - \beta \int_{\bbt^3 \times \bbr^3} | u_f - \xi|^2 f d\xi dx + 3\sigma \int_{\bbt^3 \times \bbr^3} f d\xi dx.\cr
\end{aligned}
\end{align*}
\end{proposition}
\begin{proof}
$(i)$ and $(ii)$ are readily obtained from the system \eqref{V-INS}. For the estimate of $(iii)$, we multiply $\eqref{V-INS}_1$ by $|\xi|^2/2$ and integrating over $\bbt^3 \times \bbr^3$ to get
\begin{align}\label{energy-f}
\begin{aligned}
&\frac{1}{2} \frac{d}{dt} \int_{\bbt^3 \times \bbr^3} |\xi|^2 f d\xi dx \\
&\qquad = \alpha \int_{\bbt^3 \times \bbr^3} \xi \cdot (u-\xi)f d \xi dx + \beta \int_{\bbt^3 \times \bbr^3} |u_f - \xi|^2 f d\xi dx \cr
&\qquad \quad + 3\sigma \int_{\bbt^3 \times \bbr^3} f d\xi dx, \cr
\end{aligned}
\end{align}
where we have used that
\[
\int_{\bbt^3 \times \bbr^3} u_f \cdot(u_f - \xi)f d\xi dx = 0.
\]
On the other hand, from $\eqref{V-INS}_3$, we get
\begin{equation}\label{energy-u}
\frac{1}{2} \frac{d}{dt} \int_{\bbt^3} |u|^2 dx + \mu \int_{\bbt^3} |\nabla_x u|^2 dx = - \alpha \int_{\bbt^3 \times \bbr^3} u \cdot (u-\xi)f d\xi dx.
\end{equation}
We now combine \eqref{energy-f} and \eqref{energy-u} to conclude the desired result.
\end{proof}
\begin{remark}
Throughout the paper, without loss of generality, we assume
\[
M_0(f_0)=\int_{\bbt^3 \times \bbr^3} f_0(x,\xi) dx d\xi=1.
\]
Then it follows from mass conservation (Proposition \ref{energy-prop} (i)) that
\[
M_0(f)(t)=\int_{\bbt^3 \times \bbr^3} f(x,\xi,t) dx d\xi=1, \quad t \geq 0.
\]

\end{remark}
We next provide a $L^p$-estimate for the particle density function $f$.
\begin{proposition}\label{pro:lp} Let $(f,u)$ be any smooth solutions to the system \eqref{V-INS}. Then we have
\[
\frac{d}{dt}\|f \|^p_{L^p} + \frac{4\sigma(p-1)}{p}\| \nabla_{\xi}f^{\frac{p}{2}}\|^2_{L^2} = 3(\alpha + \beta)(p-1)\|f \|^p_{L^p}.
\]
In particular, we have that
\[
\|f \|_{L^{\infty}(\bbt^3 \times \bbr^3 \times [0,T])} \leq C(T,\alpha,\beta) \|f_0\|_{L^{\infty}(\bbt^3 \times \bbr^3)}.
\]
\end{proposition}
\begin{proof}
(i) Multiplying $\eqref{V-INS}_1$ by $pf^{p-1}$ and integrate over $\bbt^3 \times \bbr^3$ to obtain
\begin{align*}
\begin{aligned}
&\frac{d}{dt} \int_{\bbt^3 \times \bbr^3} f^p dx d\xi \cr
&\quad =-\alpha p\int_{\bbt^3 \times \bbr^3} f^{p-1} \nabla_{\xi} \cdot \big( (u - \xi)f\big) dx d\xi\\
&\qquad  - \beta p\int_{\bbt^3 \times \bbr^3} f^{p-1} \nabla_{\xi} \cdot \big( (u_f - \xi)f\big) dx d\xi
 + \sigma p \int_{\bbt^3 \times \bbr^3} \Delta_{\xi} f d\xi dx \cr
&\quad =: I_1 + I_2 + I_3.
\end{aligned}
\end{align*}
For the estimates of $I_i,i=1,2,3$, it is straightforward to get by integration by parts
\begin{align*}
\begin{aligned}
I_1 
&= 3\alpha (p-1) \int_{\bbt^3 \times \bbr^3} f^p dx d\xi, \cr \\
I_2 &= 3\beta (p-1) \int_{\bbt^3 \times \bbr^3} f^p dx d\xi, \cr \\
I_3 
& = -\frac{4\sigma (p-1)}{p}\int_{\bbt^3 \times \bbr^3} |\nabla_{\xi}f^{\frac{p}{2}}|^2 dx d\xi.
\end{aligned}
\end{align*}
This concludes the proof.
\end{proof}

\subsection{Integrability and velocity averaging}
Let us now provide two useful lemmas
for later reference. For the proofs of these lemmas, we refer to \cite{BDGM,Glassey,KMT}.

\begin{lemma}\label{usef-lem-2} Let $k_2 > k_1$ and $f$ be a nonnegative function. Suppose $f$ satisfies
\[
\|f\|_{L^{\infty}(\bbt^3 \times \bbr^3 \times [0,T])} < \infty, \quad \mbox{and} \quad m_{k_2} (f)(x,t) < \infty, \quad a.e. ~(x,t).
\]
Then the following inequality holds.
\[
m_{k_1} (f)(x,t) \leq \left( \frac{4\pi}{3}\|f\|_{L^{\infty}(\bbt^3 \times \bbr^3 \times [0,T])} + 1 \right)m_{k_2}(f)(x,t)^{\frac{k_1 + 3}{k_2 + 3}}, \quad a.e. ~(x,t).
\]
\end{lemma}

We conclude this section by stating the following version
of the celebrated velocity averaging lemma.

\begin{lemma}\label{lem:velocity}
For $1 \leq p < \frac{5}{4}$, let $\{G^n\}_n$ be bounded in $L^p(\T^3 \times \R^3 \times (0,T))$. Suppose that
\begin{equation*}
f^n \mbox{ is bounded in } L^\infty(0,T;L^1 \cap L^{\infty}(\T^3 \times \R^3)),
\end{equation*}
\begin{equation*}
 |\xi|^2f^n \mbox{ is bounded in } L^\infty(0,T;L^1(\T^3 \times \R^3)).
\end{equation*}
If
$f^n$ and $G^n$ satisfy the equation
$$
f^n_t + \xi \cdot \Grad f^n = \Grad_\xi^k G^n, \qquad f^n|_{t=0} = f_0 \in L^p(\T^3 \times \R^3),
$$
for a multi-index $k$.
Then, for any $\psi(\xi)$, such that $|\psi(\xi)| \leq c|\xi|$ as $|\xi| \rightarrow \infty$, the sequence
\begin{equation*}
	\Set{\int_{\R^3}f^n \psi(\xi)~d\xi}_n,
\end{equation*}
is relatively compact in $L^p(\T^3 \times (0,T))$.
\end{lemma}

%
%
%
%
%

\section{Global existence of weak solutions (Theorem \ref{weak-thm})}\label{sec4}
In this section, we will prove the existence
of weak solutions to the system \eqref{V-INS}
and thereby prove Theorem \ref{weak-thm}.
Our strategy will be to pass to the limit in a sequence
of approximate solutions. To define the approximate solutions,
fix a small $\varepsilon>0$, let $\theta$ be a standard mollifier:
$$
\theta \geq 0, \quad \theta \in {\mathcal C}_0^\infty(\T^3),
\quad \mbox{supp}_x \theta \subset B_1(0), \quad \int_{\T^3} \theta(x) dx = 1,
$$
and set
$\theta_{\varepsilon}(x) :=  (1 /\varepsilon^{3}) \theta (x/\varepsilon).$
The approximate solutions are obtained by solving:
\begin{align}\label{eq:ap}
	\begin{split}
		\partial_t f + \xi \cdot \nabla f +  \nabla_\xi \cdot \left[f(\chi_{R}(u)-\xi)\right]
		&= - \nabla_\xi \cdot \left[f(\chi_{R}(u^{\eps}_f) - \xi)\right] + \sigma \Delta_\xi f \\
	\partial_t u + (\theta_{\eps}\star u)\cdot \Grad u + \Grad p &= \mu \Delta u + (m_f- \vr_fu)\vc{1}_R(u) \\
	\Div u &= 0,
	\end{split}
\end{align}
where $m_f = \int_{\R^3} \xi f d\xi$.
Compared to \eqref{V-INS}, we have introduced the regularizations
\begin{equation}\label{Nota:reg}
	\vc{1}_R(w) = \begin{cases}
		1, & |w| \leq R, \\
		0, & \text{otherwise}
	\end{cases}, \,\, \chi_{R}(w) = w\vc{1}_{R}(w),\,\,\mbox{and} \,\,
	u_f^\eps = \frac{m_f}{\vr_f + \eps},
\end{equation}
and in addition, we have regularized the convection velocity $\theta_\eps \star u$. Notice that we do not need the notation of $u_f$.

We shall also need to regularize the initial data:
\begin{equation}\label{ap:initial}
    u_0^\eps := \theta_\eps \star u_0, \qquad f_0^R := f_0\vc{1}_R(\xi).
\end{equation}

\begin{remark}
In our approximation scheme \eqref{eq:ap}, for simplicity, we set $\alpha = \beta = 1$. We also dropped the subscript $\eps$ and $R$, for instance $f^{\eps,R}$ or $u^{\eps,R}$ by $f$ or $u$.
\end{remark}

Before we can start sending $\eps \rightarrow 0$ and $R \rightarrow \infty$ in
\eqref{eq:ap}, we need to make sure that \eqref{eq:ap} actually admits
a weak solution. We will establish the following proposition.

\begin{proposition}\label{pro:approx} For a given $T >0$, suppose that $(f_0,u_0)$ satisfy \eqref{ap:initial}. Then there exists a weak solution $(f, u)$ to \eqref{eq:ap} in the sense of
Definition \ref{def-weak} (with $\chi_R(u_f^\eps)$ replacing $u_f$).
\end{proposition}

For the proof of Proposition \ref{pro:approx}, we will consider another decoupled system which is defined in the next subsection.

\subsection{The regularized and linearized system}
We shall prove Proposition \ref{pro:approx} using a
fixed point argument. For this purpose, we will use the space
$$
	\S := L^2(\T^3 \times (0,T)) \times L^2(\T^3 \times (0,T)).
$$
For $(w, \bar u) \in \S$ given, let $(f,u)$ be a weak solution to
\begin{align}\label{eq:ap2}
	\begin{split}
		\partial_t f + \xi \cdot \nabla f +  \nabla_\xi \cdot \left[f(\chi_{R}(w)-\xi)\right]
		&= -\nabla_\xi \cdot \left[f(\chi_{R}(\bar u) - \xi)\right] + \sigma \Delta_\xi f \\
	\partial_t u + (\theta_{\eps}\star u)\cdot \Grad u + \Grad p &= \mu \Delta u + (m_f - \vr_fw)\vc{1}_R(w) \\
	\Div u &= 0,
	\end{split}
\end{align}
and define the operator $\mt: \S \mapsto \S$ through the relation
$$
	\mt\left[w, \bar u \right] := \left[u, u_f^\eps\right] = \left[u, \frac{m_f}{\eps + \vr_f}\right].
$$
Observe that a fixed point $[u, u_f^\eps] = \mt[u, u_f^\eps]$ is also a solution of \eqref{eq:ap}.
Hence, Proposition \ref{pro:approx} follows if we are able to establish the existence
of such a fixed point. In this subsection, we shall achieve this by verifying the postulates of the
Schauder fixed point theorem.

\subsubsection{The operator $\mt[\cdot, \cdot]$ is well-defined}
\begin{lemma}\label{lem:fixed-well}
Let $(f_0,u_0)$ satisfy \eqref{ap:initial}, and assume that we are given
$(w, \bar u) \in \S$. Then there exists a unique solution $(f, u)$
of \eqref{eq:ap2} satisfying
\begin{equation}\label{fixed:1}
	\begin{split}
		\|f\|_{L^\infty(0,T;L^p(\T^3 \times \R^3))} + \|\Grad_\xi f^\frac{p}{2}\|^\frac{2}{p}_{L^2(0,T;L^2(\T^3 \times \R^3))}
		&\leq C(R,\sigma,T)\|f_0\|_{L^p(\T^3 \times \R^3)}\\
		\sup_{t \in (0,T)}\int f|\xi|^k~d\xi dx &\leq C(R,k,\sigma,T).
	\end{split}
\end{equation}
for $p \in [1, \infty]$ and all finite $k$, and moreover,
\begin{equation}\label{fixed:2}
\frac{1}{2}\|u\|_{L^\infty(0,T;L^2(\T^3))} + \mu \|\Grad u\|_{L^2(0,T;L^2(\T^3))} \leq \frac{1}{2}\|u_0\|_{L^\infty(0,T;L^2(\T^3))} + C(R,T).
\end{equation}
Here, $C(\cdot)$ denotes a generic constant depending on $\cdot$.
\end{lemma}

\begin{proof}
First, we observe that the two equations in \eqref{eq:ap2} are decoupled and
a solution can be obtained by first determining $f$ and then $u$.
Let us begin by discussing solutions to the first equation.

1. Since both $\chi_{R}(w)$ and $\chi_{R}(\bar u)$ are bounded in $L^\infty(\T^3 \times (0,T))$,
existence of a unique function $f \in \mathcal{C}(0,T;(L^1\cap L^\infty)(\T^3 \times \R^3))$ solving \eqref{eq:ap2} is by now standard
and can be found in \cite{Degond} (cf. \cite{KMT}). The $L^p$ bound in \eqref{fixed:1} can be found in \cite{KMT}. We also notice that for a smooth solution $f$ to \eqref{eq:ap2} provides
\begin{align*}
\begin{aligned}
\frac{d}{dt}\int_{\T^3 \times \R^3} f|\xi|^k~d\xi dx & \leq 		 \int_{\T^3 \times \R^3}f\left(2kR|\xi|^{k-1} + \sigma k(k+1)|\xi|^{k-2}\right)~d\xi dx \cr
&\leq C(R,k,\sigma)\left( \int_{\T^3 \times \R^3} f_0~d\xi dx + \int_{\T^3 \times \R^3} f|\xi|^k~d\xi dx \right).
\end{aligned}
\end{align*}
Since $\int_{\T^3 \times \R^3} f_0 d\xi dx < \infty$ and $\int_{\T^3 \times \R^3} f^R_0 |\xi|^k~d\xi dx < \infty$ for any finite $k$, we obtain that
\begin{equation}\label{eq:allmoment}
	\sup_{t \in (0,T)} \int_{\T^3 \times \R^3} f|\xi|^k~ dvdx < C(R, k, \sigma, T) \quad \text{for any $k$ finite}.
\end{equation}
This bound continues to hold for the unique solution $f$ of \eqref{eq:ap2}. To see this,
one can for instance localize $|\xi|^k$ as $\phi\left(|\xi|^k\right)$ where $\phi(r) = 1$
when $r \leq D$, and $\phi = 0$ when $r \geq 2D$, make the corresponding calculations
and send $D \rightarrow \infty$. This concludes the second inequality in \eqref{fixed:1}.

2. Let us now turn to the Navier-Stokes equations for $u$. First, since all (finite) moments of $f$ are bounded \eqref{eq:allmoment}, Lemma \ref{usef-lem-2} gives in particular
\begin{equation}\label{fixed:moment}
	\vr_f, ~ m_f \in L^\infty(0,T;L^2(\T^3)),
\end{equation}
where the inclusion constant depends on $R$.
Due to \eqref{fixed:moment}, we see that the right-hand side in the equation for $u$
is also in $L^\infty(0,T;L^2(\T^3))$, that is,
\begin{equation*}
		\|(m_f - \vr_f w)\vc{1}_{|w| \leq R}\|_{L^\infty(0,T;L^2(\T^3))} \leq C(R).
\end{equation*}
Standard parabolic theory then asserts the existence of an unique solution $u$
satisfying \eqref{fixed:2} (cf. \cite{GR}).
\end{proof}
From the previous lemma, it readily follows that $\mt[\cdot, \cdot]$ is well-defined
and maps into a bounded subset of $\S$.
\begin{corollary}\label{cor:fixed-bound}
There is a constant $C(R, \eps)$, such that
\begin{equation*}
	\left\|\mt[w, \bar u]\right\|_{\S} \leq C(R, \eps), \quad \forall (w, \bar u) \in \S.
\end{equation*}
\end{corollary}
\begin{proof}
By definition, we have that
\begin{equation*}
	\begin{split}
		\left\|\mt[w, \bar u]\right\|_S &\leq \|u\|_{L^2(\T^3 \times (0,T))}
		+ \frac{1}{\eps}\|m_f\|_{L^2(\T^3 \times (0,T))} \\
		&\leq C(R, \eps,T),
	\end{split}
\end{equation*}
where the last inequality is \eqref{fixed:1} and \eqref{fixed:2}.
\end{proof}
\subsubsection{The operator $\mt[\cdot, \cdot]$ is compact}
\begin{lemma}\label{lem:fixed-compact}
Let $(f_0, u_0)$ and $T$ be as in Proposition \ref{pro:approx},
and let $\{(w_n, \bar u_n)\}_{n=0}^\infty$ be an uniformly bounded sequence in $\S$. Then up to a subsequence $\{\mt[w_n,\bar{u}_n]\}_{n=0}^\infty$ converges strongly in $\S$.
\end{lemma}
\begin{proof}
Let $\{(u_n,f_n)\}_n$ be the sequence of solutions to \eqref{eq:ap2} corresponding
to $\{(w_n, \bar u_n)\}_n$. We will prove compactness of the two components
of $\mt[\cdot, \cdot]$ separately.

1. We take the first component of $\mt[w_n,\bar{u}_n]$, $\left.\mt[w_n, \bar u_n]\right|_1 = u_n.$
To show its compactness in $L^2(0,T;L^2(\T^3))$, it suffices to prove that
\begin{equation*}
\|u_n\|_{L^2(0,T;H^1)} \leq C, \quad \mbox{and} \quad \|\partial_t u_n\|_{L^2(0,T;\mathcal{V}')} \leq C,
\end{equation*}
due to the Aubin-Lions compactness lemma. \newline

\noindent $\bullet$ Estimate of $\|u_n\|_{L^2(0,T;H^1)} \leq C$: From \eqref{eq:ap2}, we get
\begin{align*}
\begin{aligned}
\frac{1}{2} \frac{d}{dt} \int_{\T^3} |u_n|^2 dx + \mu \int_{\bbt^3} |\nabla u_n|^2 dx &= -\int_{\T^3 \times \R^3} (w - \xi)f_n\vc{1}_{|w|\leq R} \cdot u_n d\xi dx \cr
&\leq \|\rho_{f_n}\|_{L^2} \|u_n\|_{L^2} + \|m_{f_n}\|_{L^2}\|u_n\|_{L^2} \cr
&\leq \big( \|\rho_{f_n}\|_{L^2} + \|m_{f_n}\|_{L^2} \big)^2 + \|u_n\|^2_{L^2}.
\end{aligned}
\end{align*}
Then it follows from Lemma \ref{usef-lem-2} that
\[
\frac{1}{2}\frac{d}{dt}\|u_n\|^2_{L^2} + \mu \|\nabla u_n\|^2_{L^2} \leq C + \|u_n\|^2_{L^2},
\]
and this yields
\begin{equation}\label{comp:u}
\|u_n\|_{L^{\infty}(0,T;L^2)} \leq C, \quad \mbox{and} \quad \|\nabla u_n \|_{L^2(0,T;L^2)} \leq C.
\end{equation}
$\bullet$ Estimate of $\|\partial_t u_n\|_{L^2(0,T;\mathcal{V}')} \leq C$: For this, it is enough to check the convection and drag force terms. For $\phi \in \mathcal{V}$, we obtain
\begin{align*}
\begin{aligned}
\left| \int_0^T \int_{\bbt^3} \left((\theta_\eps \star u_n) \cdot \nabla u_n\right) \cdot \phi \,dx dt \right| &= \left| \int_0^T \int_{\T^3} \left((\theta_\eps \star u_n) \cdot \nabla \phi\right) \cdot u_n \, dx dt \right| \cr
&\leq \int_0^T \|\theta_\eps \star u_n\|_{L^{\infty}}\|\nabla \phi\|_{L^2} \|u_n\|_{L^2} dt \cr
&\leq \|\theta_\eps\|_{L^2}\|u_n\|^2_{L^{\infty}(0,T;L^2)}\int_0^T \|\nabla \phi\|_{L^2} dt\cr
&\leq C(T,\varepsilon)\|\nabla \phi\|_{L^2(0,T;L^2)},
\end{aligned}
\end{align*}
by \eqref{comp:u}. This implies
\[
\phi \mapsto \int_0^T \int_{\bbt^3} \left((\theta_\eps \star u_n) \cdot \nabla u_n \right) \cdot \phi \, dx dt \quad \mbox{is bounded in} \,\,L^{2}(0,T;\mathcal{V}').
\]
For the drag force term, we obtain
\begin{align*}
\begin{aligned}
& \left| \int_0^T \int_{\T^3 \times \R^3} (w - \xi)f_n\vc{1}_{|w|\leq R}  \cdot \phi \, d\xi dx dt\right| \cr
& \qquad \leq R\|\phi\|_{L^{2}(0,T;L^2)}\|\rho_{f_n}\|_{L^2(0,T;L^2)} + \|\phi\|_{L^2(0,T;L^5)}\|m_{f_n}\|_{L^2(0,T;L^{\frac{5}{4}})} \\
&\qquad \leq C\|\phi\|_{L^2(0,T;H^1)}.
\end{aligned}
\end{align*}
Here we used again Lemma \ref{usef-lem-2} and $\T^3$ is bounded.
Hence we conclude that the drag force term is uniformly bounded in $L^2(0,T;\mathcal{V}')$.

2. The second component of $\mt[\cdot, \cdot]$ is given by
\begin{equation*}
	\left.\mt[w_n, \bar u_n]\right|_2 = \frac{m_{f_n}}{\eps + \vr_{f_n}},
\end{equation*}
and hence strong convergence follows if we can prove the
compactness of $\vr_{f_n}$ and $m_{f_n}$.
From \eqref{fixed:moment}, we have that
\begin{equation*}
	\vr_{f_n}, ~m_{f_n} \in_b L^2(\T^3 \times (0,T)),
\end{equation*}
where $\in_b$ means that the inclusion constant is independent of $n$. To show the compactness, we write \eqref{eq:ap2} in the form
\begin{equation*}
	\begin{split}
		\partial_t f_n + \xi \cdot \Grad f_n = \nabla_\xi \cdot G_n + \sigma \Delta_\xi f_n,
	\end{split}
\end{equation*}
where we have introduced the quantity
$$
	G_n = f_n\left(\chi_{R}(w_n) + \chi_{R}(\bar u_n) - 2\xi\right).
$$
For any finite $2\leq p < \infty$, an application of the H\"older inequality provides
\begin{equation*}
	\begin{split}
		&\|G_n\|_{L^p(\T^3 \times R^3)} \\
		&\quad \leq C(R)\|f_n\|_{L^p(\T^3 \times \R^3)}
		+2\|f_n\|_{L^\infty(\T^3 \times \R^3)}^\frac{p-1}{p}\left(\int_{\R^3\times \T^3}|f_n||\xi|^p~d\xi dx\right)^\frac{1}{p}
		 \leq C(R),
	\end{split}
\end{equation*}
where the last inequality is \eqref{fixed:1}-\eqref{fixed:2}. Hence, we can conclude that
$$
	G_n \in_b L^p(\T^3 \times \R^3 \times (0,T)) \quad \text{for all $p \in (1, \infty)$}.
$$
The velocity averaging Lemma \ref{lem:velocity} is then applicable
and yields that $\{\vr_{f_n}\}_n$ and $\{m_{f_n}\}_n$ are relatively compact in $L^2(\T^3 \times (0,T))$.
%
This concludes the proof of compactness of the operator $\mt$.
\end{proof}

\begin{proof}[Proof of Proposition \ref{pro:approx}]
Through Lemma \ref{lem:fixed-well}, Corollary \ref{cor:fixed-bound},
and Lemma \ref{lem:fixed-compact}, we have established that
the operator $\mt[\cdot, \cdot\,]$ is well-defined, bounded, and compact.
Moreover, continuity of the operator $\mt[\cdot, \cdot\,]$ is straightforward.
As a consequence, the postulates of the Schauder fixed point theorem
are satisfied, and hence yields the existence of a fixed point.
This concludes our proof of Proposition \ref{pro:approx}.
\end{proof}

\subsection{Uniform bounds}
To consider vanishing approximation
parameters, we will need some
uniform (in $\eps$ and $R$) $L^p$ and energy bounds
on solutions of \eqref{eq:ap}. We recall that the energy
is given by
\begin{equation*}
	\mathbb{E}(t) = \int_{\T^3 \times \R^3} f\frac{|\xi|^2}{2}\,d\xi dx +  \int_{\T^3} \frac{|u|^2}{2}\, dx
\end{equation*}

\begin{lemma}\label{lem:}
Under the conditions of Proposition \ref{pro:approx}, there
exists a constant $C>0$, independent of $R$ and $\eps$,
such that
\begin{equation}\label{uni:lp}
	\|f\|_{L^\infty(0,T;L^p(\T^3\times \R^3))}
	+ \|\Grad_\xi f^\frac{p}{2}\|_{L^2(\T^3\times \R^3 \times (0,T))}^\frac{2}{p}
	\leq C(p,\sigma,T)\|f_0\|_{L^p(\T^3 \times \R^3)},
\end{equation}
\begin{equation}\label{uni:ent}
	\begin{split}
		\sup_{t \in (0,T)} \mathbb{E}(t)
		&+ \mu \int_0^T \|\Grad u \|_{L^2(\T^3)}^2~dt + \int_{\T^3\times \R^3}f\left|\chi_{R}(u)- \xi\right|^2~d\xi dx \\
		&\leq \mathbb{E}(0)+ 3\sigma M_0(f_0)T.
	\end{split}
\end{equation}
\end{lemma}

\begin{proof}
By direct calculation using \eqref{eq:ap}, we deduce
\begin{equation*}
	\begin{split}
		&\frac{d}{dt}\left\|f\right\|_{L^p(\T^3 \times \R^3)}^p
		+ \frac{4\sigma(p-1)}{p}\|\Grad_\xi f^\frac{p}{2}\|_{L^2(\T^3\times \R^3)}^2 \\
		&\quad = \int_{\T^3 \times \R^3} (\chi_{R}(u) +  \chi_{R}(u_f^\eps)-2\xi)\cdot\Grad_\xi f^p~d\xi dx
		= 6\|f\|_{L^p(\T^3\times \R^3)}^p,
	\end{split}
\end{equation*}
and \eqref{uni:lp} follows from the Gronwall inequality.

Next, we calculate $\mathbb{E}^{\prime}(t)$ using both equations in \eqref{eq:ap};
\begin{equation}\label{uni:en1}
	\begin{split}
		\mathbb{E}^{\prime} &= \int_{\T^3}u_t \cdot u~dx + \int_{\T^3\times \R^3}f_t \frac{|\xi|^2}{2}~d\xi dx \\
		&= -\mu \|\Grad u\|_{L^2(\T^3)}^2 - \int_{\T^3}\left((\theta_{\eps}\star u)\cdot \Grad u\right)\cdot u ~dx
		+\int_{\T^3}(m_f - \vr_f u)\vc{1}_R(u)\cdot u~dx\\
		&\quad +3\sigma M_0(f_0) +  \int_{\T^3\times \R^3}f(\chi_{R}(u)- \xi)\cdot \xi~d\xi dx \\
		&\quad+  \int_{\T^3\times \R^3}f(\chi_{R}(u_f^\eps)- \xi)\cdot \xi~d\xi dx.
	\end{split}
\end{equation}
By adding and subtracting, we deduce that
\begin{equation}\label{uni:en2}
	\begin{split}
		&\int_{\T^3\times \R^3}f(\chi_{R}(u)- \xi)\cdot\xi~d\xi dx \\
		&\qquad
		= -\int_{\T^3\times \R^3}f|\chi_{R}(u)- \xi|^2~d\xi dx + \int_{\T^3}(\vr_f\chi_{R}(u) - m_f)\cdot\chi_{R}(u)~dx \\
		&\qquad =-\int_{\T^3\times \R^3}f|\chi_{R}(u)- \xi|^2~d\xi dx + \int_{\T^3}(\vr_f u - m_f)\vc{1}_R(u)\cdot u~dx.
	\end{split}
\end{equation}
We also have
\begin{equation}\label{uni:en3}
\int_{\T^3\times \R^3}f(\chi_{R}(u_f^\eps)- \xi)\cdot\xi~d\xi dx \leq \int_{\T^3} \frac{|m_f|^2}{\rho_f + \eps} \,dx - \int_{\T^3 \times \R^3} |\xi|^2 f \,d\xi dx \leq 0,
\end{equation}
where we used $|m_f|^2 \leq \rho_f \left(\int_{\R^3} |\xi|^2 f d\xi\right)$. 
By applying \eqref{uni:en2} and \eqref{uni:en3} in \eqref{uni:en1}, we obtain \eqref{uni:ent}.
\end{proof}

\subsection{The $R \rightarrow \infty$ limit}
We are now ready to send $R \rightarrow \infty$
in our approximate equation \eqref{eq:ap}.
We begin by deriving some compactness properties.
Some of the arguments we shall use in this regard
are similar to those of \cite{KMT}.
\newcommand{\lt}{\left}
\newcommand{\rt}{\right}
\begin{lemma}\label{lem:compact1}
	Let $\eps > 0$ be fixed, set $R = n$, and let $\{(f_n,u_n)\}_{n=0}^\infty$ be
	the corresponding sequence of solutions to \eqref{eq:ap}. Then, up to a subsequence as  $n \rightarrow \infty$, we have
	\begin{equation}\label{eq:conv}
		\begin{split}
			f_n &\weak f\quad \text{in $\mathcal{C}(0,T;L^p(\T^3\times \R^3))$, $p \in (1, \infty)$}, \\
			\vr_{f_n} &\rightarrow \vr_f\quad \text{a.e and in $L^p(\T^3 \times (0,T))$, $~p \in \lt(1, \frac{5}{4}\rt)$}, \\
			m_{f_n} &\rightarrow m_f \quad \text{a.e and in $L^q(\T^3 \times (0,T))$, $~q \in \lt(1, \frac{5}{4}\rt)$}, \\
			u_n &\rightarrow u \quad \text{a.e and in $L^2(\T^3 \times (0,T))$},
		\end{split}
	\end{equation}
	where $\vr_f = \int_{\R^3}f\,d\xi$ and $m_f = \int_{\R^3} \xi f \, d\xi$.
	
\end{lemma}
\begin{proof}
	1. We first apply the previous lemma and Lemma \ref{usef-lem-2}, to deduce that
	\begin{equation}\label{conv:mom}
		\begin{split}
			\vr_{f_n} \in_b L^p(\T^3 \times (0,T)), \qquad m_{f_n} \in_b L^q(\T^3 \times (0,T)),
		\end{split}
	\end{equation}
	for any $p\in \left(1, \frac{5}{3}\right)$ and $q \in \left(1, \frac{5}{4}\right)$. Using this, we apply the H\"older inequality
    to find that
    \begin{equation}\label{eq:long}
        \begin{split}
            \left|\int_0^T\int_{\T^3} \partial_t u_n \cdot v~dxdt\right|
            & \leq \| u_n\|_{L^3(\T^3 \times (0,T))}\|\Grad u_n\|_{L^2(\T^3 \times (0,T))}\|v\|_{L^6(\T^3 \times (0,T))} \\
            &\qquad + \mu \|\Grad u_n\|_{L^2(\T^3 \times (0,T))}\|\Grad v\|_{L^2(\T^3 \times (0,T))} \\
            &\qquad + \|p_n\|_{L^2(\T^3 \times (0,T))}\|\Div v\|_{L^2(\T^3 \times (0,T))} \\
            &\qquad + \|m_{f_n}\|_{L^\infty(0,T;L^q(\T^3))}\|v\|_{L^1(0,T;L^\frac{q}{q-1}(\T^3))} \\
            &\qquad + \|\vr_{f_n}u_{n}\|_{L^2(0,T;L^q(\T^3))}\|v\|_{L^2(0,T;L^\frac{q}{q-1}(\T^3))},
        \end{split}
    \end{equation}
	for $q <\frac{5}{4}$ and all $v \in [\mathcal{C}_0^\infty(\T^3 \times (0,T))]^3$. To bound the last norm in the right-hand side, we shall need the calculation
	\begin{equation}\label{conv:1}
		\begin{split}
			\|\vr_{f_n} u_n\|_{L^2(0,T; L^q(\T^3))} &\leq \|\vr_{f_n}\|_{L^\infty(0,T;L^{\frac{5q}{5-q}}(\T^3))}\|u_n\|_{L^2(0,T; L^5(\T^3))} \\
			&\leq C\|\vr_{f_n}\|_{L^\infty(0,T;L^{\frac{5q}{5-q}}(\T^3))}\|u_n\|_{L^2(0,T; W^{1,2}(\T^3))}.
		\end{split}
	\end{equation}
We notice that $\frac{5q}{5 - q} < \frac{5}{3}$ since $q < \frac{5}{4}$. By applying \eqref{conv:1} in \eqref{eq:long}, using \eqref{conv:mom} and \eqref{uni:ent}, and using $q = \frac{6}{5}$, we deduce that
    $$
        \partial_t u_n \in_b L^\frac{6}{5}(0,T; W^{-1,2}(\T^3)).
    $$
    Since in addition $u_n \in_b L^2(0,T;W^{1,2}(\T^3))$,
	we can apply the Aubin-Lions lemma to conclude
	\begin{equation*}
		u_n \rightarrow u \quad \text{a.e  and in $L^2(0,T;L^2(\T^3))$ }, \quad \text{as $n \rightarrow \infty$}.
	\end{equation*}

	2. To conclude compactness of $\vr_{f_n}$, we write
	the first equation in \eqref{eq:ap} in the form
	\begin{equation*}
		\partial_t f_n + \xi \cdot \Grad f_n = \Grad_\xi \cdot G_n + \sigma \Delta_\xi f_n,
	\end{equation*}
	where $G_n = f_n \left( \chi_n(u_n)+\chi_n(u_{f_n}^\eps)-2\xi\right)$. By direct calculation,
	using the uniform bounds \eqref{uni:lp}, \eqref{uni:ent}, and \eqref{conv:1}, we deduce
	\begin{equation*}
		\begin{split}
		&\|G_n\|_{L^q(\T^3 \times \R^3 \times (0,T))} \\
		&~\leq \|f_n u_n\|_{L^q(\T^3 \times \R^3 \times (0,T))} + \|f_n u_{f_n}\|_{L^q(\T^3 \times \R^3 \times (0,T))}
		 + 2\|f_n |\xi|\|_{L^q(\T^3 \times \R^3 \times (0,T))} \\
		 &~ \leq C\left(\|u_n\|_{L^q(0,T; L^5(\T^3))} +\|m_{f_n}\|_{L^q(\T^3 \times (0,T))}
		 + \|\sqrt{f_n}|\xi|\|_{L^2(\T^3 \times \R^3 \times (0,T))}  \right) \\
         &~\leq C,
		\end{split}
	\end{equation*}
where we used
\begin{align}\label{conv:1.5}
\begin{aligned}
\|f_u u_n\|_{L^q(\T^3 \times \R^3)} &\leq \|f_n\|^{\frac{q-1}{q}}_{L^\infty(\T^3 \times \R^3)}\|\rho_{f_n}^{\frac{1}{q}}\|_{L^{\frac{5q}{5-q}}(\T^3)}\|u_n\|_{L^5(\T^3)} \cr &\leq \|f_n\|^{\frac{q-1}{q}}_{L^\infty(\T^3 \times \R^3)}\|\rho_{f_n}\|^{\frac{1}{q}}_{L^{\frac{5}{5-q}}(\T^3)}\|u_n\|_{L^5(\T^3)}, 
\end{aligned}
\end{align}
and here $\|\rho_{f_n}\|^{\frac{1}{q}}_{L^{\frac{5}{5-q}}(\T^3)}$ is uniformly bounded in $n$ since $\frac{5}{5-q} < \frac{5}{3}$.
Hence we can conclude that
\begin{equation}\label{conv:2}
G_n \in_b L^q(\T^3 \times \R^3 \times (0,T)), \quad \forall q \in \left(1, \frac{5}{4}\right).
\end{equation}	
The velocity averaging Lemma \ref{lem:velocity} is then applicable and yields
\begin{equation}\label{conv:VAL}
\int_{\R^3}f_n \psi(\xi)~d\xi \rightarrow \int_{\R^3}f \psi(\xi)~ d\xi\quad \text{in $L^q(\T^3 \times (0,T))$},
\end{equation}
for any $\psi(\xi)$ such that $|\psi(\xi)| \leq c|\xi|$ as $|\xi| \rightarrow \infty$, and any $q \in \lt(1, \frac{5}{4}\rt)$.
If we set $\psi(\xi) \equiv 1$ and $\psi(\xi) \equiv \xi$ in \eqref{conv:VAL}, we obtain
\begin{equation*}
\vr_{f_n} \rightarrow \vr_f \quad \mbox{and} \quad m_{f_n} \rightarrow m_f \quad\text{in $L^q(\T^3 \times (0,T))$}.
\end{equation*}

The proof is now complete.
\end{proof}
In the next lemma, we establish convergence of solutions
to \eqref{eq:ap}  as
$R \rightarrow \infty$. Specifically, we will send $n \rightarrow \infty$ in
\begin{align}\label{eq:apn}
	\begin{split}
		\partial_t f_n + \xi \cdot \nabla f_n +  \nabla_\xi \cdot \left[f_n(\chi_{n}(u_n)-\xi)\right]
		&= -\nabla_\xi \cdot \left[f_n(\chi_{n}(u^{\eps}_{f_n}) - \xi)\right] + \sigma \Delta_\xi f_n \\
	\partial_t u_n + (\theta_{\eps}\star u_n)\cdot \Grad u_n + \Grad p_n &= \mu \Delta u_n + (m_{f_n} - \vr_{f_n}u_n)\vc{1}_n(u_n) \\
	\Div u_n &= 0.
	\end{split}
\end{align}

\begin{lemma}\label{lem:Rlim}
Under the conditions of the previous lemma,  $(f,u)$ is a weak solution of
\begin{align}\label{eq:ape}
	\begin{split}
		\partial_t f + \xi \cdot \nabla f +  \nabla_\xi \cdot \left[f(u-\xi)\right]
		&= -\nabla_\xi \cdot \left[f(u^{\eps}_f - \xi)\right] + \sigma \Delta_\xi f \\
	\partial_t u + (\theta_{\eps}\star u)\cdot \Grad u + \Grad p &= \mu \Delta u + (m_f - \vr_fu) \\
	\Div u &= 0,
	\end{split}
\end{align}
in the sense of Definition \ref{def-weak}, where  $u_f^\eps$ is defined in \eqref{Nota:reg}. Moreover, $(f, u)$ satisfies
\begin{equation}\label{uni2:lp}
	\|f\|_{L^\infty(0,T;L^p(\T^3\times \R^3))}
	+ \|\Grad_\xi f^\frac{p}{2}\|_{L^2(0,T;L^2(\T^3\times \R^3))}^\frac{2}{p}
	\leq C(p,T)\|f_0\|_{L^p(\T^3 \times \R^3)}
\end{equation}
\begin{equation}\label{uni2:ent}
	\begin{split}
		\sup_{t \in (0,T)} \mathbb{E}(t)
		&+ \mu \int_0^T \|\Grad u \|_{L^2(\T^3)}^2~dt + \int_{\T^3\times \R^3}f\left|u- \xi\right|^2\,d\xi dx \\
		& 
		\leq  \mathbb{E}(0)+ 3\sigma M_0(f_0)T.
	\end{split}
\end{equation}
\end{lemma}
\begin{proof}
The only problematic terms when passing to the limit in \eqref{eq:apn}, are
\begin{equation}\label{conv:prod}
	f_n\chi_{n}(u_n), \quad f_n\chi_{n}(u_{f_n}^{\eps}), \quad\vr_{f_n}u_n\vc{1}_n(u_n).
\end{equation}

1. Let us begin with the latter. From \eqref{conv:1}, we have that
\begin{equation}\label{comple1}
	\|\vr_{f_n}u_n\|_{L^2(0,T; L^q(\T^3))} \leq C, \quad q < \frac{5}{4}.
\end{equation}
As a consequence, we can apply weak compactness to \eqref{comple1}, \eqref{uni:ent} and use the strong convergences of $\rho_{f_n}$ and $u_n$ in \eqref{eq:conv} to deduce
\begin{equation}\label{comple2}
	\vr_{f_n} u_n \weak \vr_fu \quad \text{as } n \to \infty \text{ in $L^2(0,T; L^q(\T^3))$, \quad $q < \frac{5}{4}$}.
\end{equation}
Now, by adding and subtracting, we see that
\begin{equation}\label{conv:NSt1}
	\vr_{f_n} u_n\vc{1}_n(u_n) = \vr_{f_n} u_n - \vr_{f_n} u_n(1 - \vc{1}_n(u_n)),
\end{equation}
where the last term converges to zero as
\begin{equation*}
	\begin{split}
	&\|\vr_{f_n} u_n(1 - \vc{1}_n(u_n))\|_{L^1(\T^3 \times (0,T))}  \\
	&\qquad \leq C\|\vr_{f_n}u_n\|_{L^2(0,T: L^\frac{6}{5}(\T^3))}\|1 - \vc{1}_n(u_n)\|_{L^{2}(0,T; L^{6}(\T^3))} \\
	&\qquad \leq \frac{C}{n}\|u_n\|_{L^{2}(0,T; L^6(\T^3))} \leq \frac{C}{n}\|u_n\|_{L^{2}(0,T; H^1(\T^3))}\overset{n\rightarrow \infty}{\longrightarrow} 0,
	\end{split}
\end{equation*}
where we have used $q=\frac{6}{5}< \frac{5}{4}$ in \eqref{conv:1} and the estimates of uniform bounds for the approximations \eqref{uni:ent}.
Hence, we can send $n \rightarrow \infty$ in \eqref{conv:NSt1} to conclude
\begin{equation}\label{conv:prod1}
	\vr_{f_n} u_n\vc{1}_n(u_n) \weak \vr_fu \quad \text{as } n \to \infty\text{ in $L^2(0,T; L^q(\T^3))$, \quad $q < \frac{5}{4}$}.
\end{equation}

2. Let us now consider the first  term in \eqref{conv:prod}. For this purpose,
we use \eqref{uni:ent} and \eqref{conv:1.5} to find
\begin{align*}
\begin{aligned}
&\|f_u u_n\|_{L^2(0,T;L^q(\T^3 \times \R^3))} \cr
&\quad \leq \|f_n\|^{\frac{q-1}{q}}_{L^\infty(\T^3 \times \R^3 \times (0,T))}\|\rho_{f_n}\|^{\frac{1}{q}}_{L^\infty(0,T;L^{\frac{5}{5-q}}(\T^3))}\|u_n\|_{L^2(0,T;L^5(\T^3))}\cr
&\quad  \leq C \|u_n\|_{L^2(0,T;H^1(\T^3))} \leq C.
\end{aligned}
\end{align*}
Then we use the similar argument to \eqref{comple2} to obtain
\begin{equation}\label{conv:weak1}
	f_n u_n \weak fu \text{ in $L^2(0,T;L^q(\T^3\times \R^3))$}.
\end{equation}
Next, by adding and subtracting, we write
\begin{equation}\label{conv:weak2}
	f_n \chi_n(u_n) = f_n u_n + f_n(\chi_n(u_n) - u_n),
\end{equation}
where the last term converges to zero as
\begin{equation*}
	\begin{split}
		&\|f_n(\chi_n(u_n) - u_n )\|_{L^1((0,T)\times \T^3 \times \R^3)}\\
		&\qquad = \|\vr_{f_n} u_n (1 - \vc{1}_n(u_n))\|_{L^1(0,T)\times \T^3)}	 \\
		&\qquad \leq \|\vr_{f_n}\|_{L^\infty(0,T;L^\frac{3}{2}(\T^3))}\|u_n\|_{L^2(0,T;L^6(\T^3))}
					\|1 - \vc{1}_n(u_n)\|_{L^2(0,T;L^6(\T^3))} \\
	&\qquad \leq \frac{1}{n}\|\vr_{f_n}\|_{L^\infty(0,T;L^\frac{3}{2}(\T^3))}\|u_n\|_{L^2(0,T;L^6(\T^3))}^2
	\overset{n \rightarrow \infty}{\longrightarrow }0.
	\end{split}
\end{equation*}
This, together with \eqref{conv:weak1}, in \eqref{conv:weak2} yields
\begin{equation}\label{conv:prod2}
	f_n\chi_n(u_n) \weak fu \text{ in $L^2(0,T;L^q(\T^3\times \R^3))$}.
\end{equation}

3. Finally, we consider the second term in \eqref{conv:prod}.
First, we calculate
\begin{equation*}
	\begin{split}
	\|f_n u_{f_n}^{\eps}\|_{L^\infty(0,T;L^q(\T^3\times \R^3))}
	&\leq \|f_n\|_{L^\infty(0,T;L^\infty(\T^3 \times \R^3))}\|u_{f_n}^\eps\|_{L^\infty(0,T;L^q(\T^3))} \\
    &\leq \frac{C}{\eps}\|m_{f_n}\|_{L^\infty(0,T;L^q(\T^3))} \leq \frac{C}{\eps}.
	\end{split}
\end{equation*}
where the last inequality is \eqref{uni:lp}, \eqref{conv:mom}, and $q < \frac{5}{4}$. We also notice that the convergence estimates of $\rho_{f_n}$ and $m_{f_n}$ in \eqref{eq:conv} and $\frac{1}{\eps + \rho_{f_n}} \leq \frac1\eps$ yield 
\begin{equation*}
	u_{f_n}^{\eps} = \frac{m_{f_n}}{\eps + \vr_{f_n}}
	\overset{n \rightarrow \infty}{\longrightarrow} \frac{m_f}{\eps + \vr_f}
	\text{ in $L^{q}(\T^3 \times (0,T))$}, \quad q < \frac54,
\end{equation*}
for each fixed $\eps > 0$, due to a simple application of Vitali's convergence theorem. In particular, we again use the similar strategy to \eqref{comple2} to have that
\begin{equation}\label{conv:last1}
	f_n u_{f_n}^{\eps} \weak f u_f^\eps \quad \text{as $n \rightarrow \infty$ in $L^{q}(\T^3 \times (0,T))$}.
\end{equation}
By adding and subtracting,
\begin{equation}\label{conv:last2}
	f_n\chi_n(u_{f_n}^{\eps})= f_nu_{f_n}^{\eps} + f_n(\chi_n(u_{f_n}^{\eps}) - u_{f_n}^{\eps}),
\end{equation}
where the last term satisfies
\begin{equation}\label{conv:last3}
	\begin{split}
		&\|f_n(\chi_n(u_{f_n}^{\eps}) - u_{f_n}^{\eps})\|_{L^1(\T^3 \times \R^3 \times (0,T))} \\
		&= \|\vr_{f_n} u_{f_n}^\eps(1 - \vc{1}_n(u_{f_n}^\eps))\|_{L^1(\T^3 \times (0,T))}\leq \frac{1}{n}\int_0^T\int_{\T^3}\vr_{f_n} |u_{f_n}^\eps|^2~dxdt
		\overset{n\rightarrow \infty}{\longrightarrow} 0,
	\end{split}
\end{equation}
since $\int_{\T^3}\vr_{f_n} |u_{f_n}^\eps|^2~dx \leq \int_{\T^3 \times \R^3}f_n|\xi|^2~d\xi dx$, which is bounded by \eqref{uni:ent}.
By combining \eqref{conv:last1}, \eqref{conv:last2}, and \eqref{conv:last3}, we conclude
\begin{equation}\label{conv:prod3}
	f_n\chi_n(u_{f_n}^{\eps}) \rightarrow fu_f^\eps  \quad \text{as $n \rightarrow \infty$ in $L^{q}(\T^3 \times (0,T))$}.
\end{equation}

4. Equipped with \eqref{conv:prod1}, \eqref{conv:prod2}, and \eqref{conv:prod3}, there are no
problems with passing to the limit in \eqref{eq:apn} to conclude that $(u,f)$ is a weak
solution to \eqref{eq:ape}. The bounds \eqref{uni2:lp} and \eqref{uni2:ent}
can be proved as in Propositions \ref{energy-prop} and \ref{pro:lp}.

\end{proof}

\subsection{The $\eps \rightarrow 0$ limit and proof of Theorem \ref{weak-thm}}
We will now send $\eps \rightarrow 0$ in \eqref{eq:ape} and thereby
conclude the proof of Theorem \ref{weak-thm}. The largest challenge is presented by possible
vacuum regions of $\vr_f$ rendering passing
to the limit in $u_f$ non-trivial.
 To this end,
we shall need the following lemma:

\begin{lemma}
Let $\{(f_\eps,u_\eps)\}_{\eps>0}$ be a sequence of weak solutions to \eqref{eq:ape}.
As $\eps \rightarrow 0$,
	\begin{equation}\label{eq:conv2}
		\begin{split}
			f_\eps &\weak f\quad \text{in $\mathcal{C}(0,T;L^p(\T^3\times \R^3))$, $p \in (1, \infty)$}, \\
			\vr_{f_\eps} &\rightarrow \vr_f\quad \text{a.e and in $L^p(\T^3 \times (0,T))$, $~p \in \lt(1, \frac{5}{4}\rt)$}, \\
			m_{f_\eps} &\rightarrow m_f \quad \text{a.e and in $L^q(\T^3 \times (0,T))$, $~q \in \lt(1, \frac{5}{4}\rt)$}, \\
			u_\eps &\rightarrow u \quad \text{a.e and in $L^2(\T^3 \times (0,T))$},
		\end{split}
	\end{equation}
where $\vr_f = \int_{\R^3}f\,d\xi$, $m_f = \int_{\R^3} \xi f\,d\xi$, and where the convergence may take place along a subsequence.
\end{lemma}

\begin{proof}
Since \eqref{conv:mom}, \eqref{conv:1}, and \eqref{conv:2} hold independently of $\eps$,
the proof follows by the exact same arguments as the proof of Lemma \ref{lem:compact1}.
\end{proof}

Theorem \ref{weak-thm} follows as a consequence of the following lemma.
\begin{lemma}
Under the conditions of the previous lemma, $(f,u)$ is a weak solution
of \eqref{V-INS} in the sense of Definition \ref{def-weak}, where

	$$
		u_f = \begin{cases}
			\displaystyle\frac{\int_{\R^3}f\xi ~d\xi}{\int_{\R^3}f~d\xi}, & \vr_f \neq 0, \\
			0, & \vr_f = 0,
		\end{cases}
	$$

\end{lemma}

\begin{proof}
From \eqref{eq:conv2}, we easily conclude that
\begin{equation*}
	\begin{split}
		f_\eps(u_\eps - \xi) &\weak f(u-\xi) \text{ in $L^2(0,T;L^q(\T^3 \times \R^3))$},\quad q < \frac{5}{4} \\
		m_{f_\eps} - \vr_{f_\eps}u_\eps & \weak m_f - \vr_fu \text{ in $L^2(0,T;L^q(\T^3))$}, \quad q < \frac{5}{4}\\
		 (\theta_\eps \star u_\eps)\cdot \Grad u_\eps & \weak u\cdot \Grad u
		\text{ in $L^\frac{6}{5}(\T^3 \times (0,T))$}.
	\end{split}
\end{equation*}
Hence, in order to pass to the limit in \eqref{eq:ape}, it remains to prove that
$$
	f_\eps u_{f_\eps}^{\eps} \to fu_f \quad \text{as $\eps \rightarrow 0$} \mbox{ in the sense of distribution}.
$$
For this purpose, let $\lambda > 0$ be a small parameter and define
$$
	A_\lambda = \Set{(x,t): \vr_f(x,t)> \lambda}.
$$
Since $\vr_{f_\eps} \rightarrow \vr_f$ a.e, Egoroff's theorem yields,
for any $\eta > 0$, the existence of a set $B_{\lambda,\eta}$ with $|A_\lambda \setminus B_{\lambda,\eta}| < \eta$
and where $\vr_{f_\eps} \rightarrow \vr_f \text{ uniformly on $B_{\lambda,\eta}$}$. In particular, for
a sufficiently small $\bar \eps$,
$$
	\vr_{f_\eps} > \lambda - \frac{\eta}{2}, \quad \forall \eps < \bar \eps, \quad (x,t)\in B_{\lambda,\eta}.
$$
By virtue of \eqref{eq:conv2}, we have that
\begin{equation*}
	u_{f_\eps}^{\eps} = \frac{m_{f_\eps}}{\eps + \vr_{f_\eps}}
	\rightarrow \frac{m_f}{\vr_f} \text{ in $L^q(B_{\lambda,\eta})$}, \quad q < \frac{5}{4}.
\end{equation*}
In particular, since $f_\eps$ converges weakly, we can conclude that
\begin{equation*}
	f_\eps u_{f_\eps}^{\eps} \weak fu_f \quad \text{in $L^q(B_{\lambda,\eta} \times \R^3)$ as $\eps \rightarrow 0$}.
\end{equation*}
We then write
\begin{equation*}
	f_\eps u_{f_\eps}^{\eps}\vc{1}_{A_\lambda} = f_\eps u_{f_\eps}^{\eps}\vc{1}_{B_\eta} + f_\eps u_{f_\eps}^{\eps}\vc{1}_{A_\lambda \setminus B_\eta},
\end{equation*}
where we see that the last term is small due to the following bound
\begin{equation}\label{new401}
	\left\|f_\eps u_{f_\eps}^{\eps}\vc{1}_{A_\lambda \setminus B_\eta}\right\|_{L^1(\T^3\times \R^3 \times (0,T))}
	\leq \eta^\frac{1}{q'}\|m_{f_\eps}\|_{L^q(\T^3 \times (0,T))}
	= \mathcal{O}(\eta^\frac{1}{q'}).
\end{equation}
Since we can choose $\eta$ arbitrarily small, we must have that
\begin{equation*}
	f_\eps u_{f_\eps}^{\eps} \weak fu_f \quad \text{in $L^q(A_\lambda \times \R^3)$ as $\eps \rightarrow 0$}.
\end{equation*}
For the estimate on the set $(\T^3 \times (0,T))\setminus A_\lambda$, we let $\eta$ be a  small parameter and make another application of Egoroff's theorem to obtain a set $C_{\lambda,\eta}$ such that
$$
\left|\left((\T^3 \times (0,T))\setminus A_\lambda\right)\setminus C_{\lambda,\eta} \right| < \eta,
\qquad \vr_{f_\eps} < \lambda+ \frac{\eta}{2} \quad \forall \,\eps < \overline{\eps}.
$$
On $C_{\lambda,\eta}$, the product $\vr_{f_n} u_{f_n}$ is controlled by $\lambda + \frac{\eta}{2}$
as
\begin{equation*}
	\begin{split}
		 \left\|f_\eps u_{f_\eps}^\eps\right\|_{L^1(C_{\lambda,\eta} \times \R^3)}
		&\leq \left(\int_{C_{\lambda,\eta}} \vr_{f_\eps} dx dt\right)^\frac{1}{2}
		\left(\int_0^T \int_{\T^3}\vr_{f_\eps}\left|u_{f_\eps}^\eps\right|^2~dxdt\right)^\frac{1}{2}\\
		&\leq \left(\lambda + \frac{\eta}{2}\right)^\frac{1}{2}C,
	\end{split}
\end{equation*}
where we have used $\int_{\T^3}\vr_{f_\eps} |u_{f_\eps}^\eps|^2\,dx \leq \int_{\T^3 \times \R^3}f_\eps|\xi|^2\,d\xi dx$ which is bounded by \eqref{uni:ent}, and the fact that $|\T^3 \times (0,T)|$ is finite to conclude the last inequality.
As in \eqref{new401}, we also see that
$$
\left\|f_\eps u_{f_\eps}^{\eps}\right\|_{L^1\left(\left(((\T^3 \times (0,T))\setminus A_\lambda\right)\setminus C_{\lambda,\eta}) \times \R^3\right)}
\leq \mathcal{O}\left(\eta^\frac{1}{q'}\right).
$$
Since $\eta$ can be chosen arbitrarily small, we deduce
\begin{equation*}
	\left\|f_\eps u_{f_\eps}^{\eps}\right\|_{L^1\left( \left((\T^3 \times (0,T)) \setminus A_\lambda\right) \times \R^3 \right)} \leq \mathcal{O}\left(\lambda^\frac{1}{2}\right).
\end{equation*}
Hence by choosing sufficiently small $\lambda$ to conclude
$$
	f_\eps u^{\eps}_{f_\eps} \weak f u_f \text{ in $L^q(\T^3 \times \R^3 \times (0,T))$}, \quad q < \frac{5}{4}.
$$
This completes the proof.
\end{proof}

%
%
\renewcommand{\Re}{\mathcal{H}}
\section{Hydrodynamic limit (Theorem \ref{lim-thm})}
In this section, we will study the
flocking-Navier-Stokes system \eqref{V-INS} under the assumption of strong noise and local alignment.
In this regime, we shall rigorously establish
that the evolution can be accurately described
by a coupled compressible Euler and incompressible Navier-Stokes system, and thereby prove Theorem \ref{lim-thm}. For the reader's convenience, we recall the equations
under consideration:
\begin{align}\label{sing-V-INS}
\begin{aligned}
\partial_t f^\eps + \xi \cdot \nabla f^\eps + \nabla_{\xi} \cdot \left[ (u^\eps - \xi)f^\eps\right] &= \frac{1}{\eps} \nabla_\xi \cdot \left[ \nabla_\xi f^\eps - (u_{f^\eps} - \xi)f^\eps\right],\cr
\partial_t u^\eps + u^\eps \cdot \nabla u^\eps + \nabla p^\eps - \mu \Delta u^\eps &= - \int_{\R^3} (u^\eps - \xi) f^\eps d\xi,\cr
\nabla \cdot u^\eps &= 0,
\end{aligned}
\end{align}
subject to
\begin{equation}\label{ini-sing-v-ins}
(f^\eps(x,\xi,0),u^\eps(x,0)) = (f_0(x,\xi),u_0(x)).
\end{equation}
Our goal is to prove that solutions of this
system can be well approximated by the Euler-Navier-Stokes system
\begin{align*}
    \partial_t u_f + \nabla \cdot (\vr_f u_f) &= 0, \\
    (\vr_f u_f)_t + \nabla \cdot(\vr_f u_f \otimes u_f)
    + \Grad \vr_f &= \vr_f(u-u_f), \\
    u_t + u \cdot \nabla u + \nabla p - \mu \Delta u &= - \vr_f(u-u_d), \\
    \nabla \cdot u &= 0,
\end{align*}
provided $\eps$ is sufficiently small.

\subsection{Entropy of weak solutions}
We first show that the weak solutions obtained from Theorem \ref{weak-thm} satisfies some entropy inequalities
that are uniform in $\eps$. For this, we set
\begin{align}\label{hydro:nota}
\begin{aligned}
\mathcal{F}(f^\eps,u^\eps) &:= \int_{\T^3 \times \R^3} f^\eps\left( \log f^\eps + \frac{|\xi|^2}{2}\right) dx d\xi + \int_{\T^3} \frac{|u^\eps|^2}{2} dx, \cr
D_1(f^\eps)&:= \int_{\T^3 \times \R^3} \frac{1}{f^\eps} | \nabla_\xi f^\eps + \left( u_{f^\eps} - \xi \right)f^\eps |^2 dx d\xi, \cr
D_2(f^\eps,u^\eps) &:= \int_{\T^3 \times \R^3} |u^\eps - \xi|^2 f^\eps dx d\xi + \mu \int_{\T^3}|\nabla u^\eps|^2 dx.
\end{aligned}
\end{align}
Then it follows from Proposition \ref{energy-prop} that
\begin{equation*}
\frac{d}{dt}\mathcal{F}(f^\eps,u^\eps) + \frac{1}{\eps}D_1(f^\eps) + D_2(f^\eps,u^\eps) = 3M_0(f_0),
\end{equation*}
and this yields
\begin{equation}\label{hydro-est-0}
\mathcal{F}(f^\eps,u^\eps) - \mathcal{F}(f_0,u_0) + \frac{1}{\eps}\int_0^tD_1(f^\eps)ds + \int_0^tD_2(f^\eps,u^\eps)ds = 3M_0(f_0)t.
\end{equation}
On the other hand, we notice that
\[
\int_{\T^3 \times \R^3} f^\eps |\log f^\eps| dx d\xi 
\leq \int_{\T^3 \times \R^3} f^\eps \log f^\eps dx d\xi + \frac{1}{4}\int_{\T^3 \times \R^3} (1 + |\xi|^2)f^\eps dx d\xi + C.
\]
This implies that
\begin{align}\label{hydro-better-est-1}
\begin{aligned}
&\int_{\T^3 \times \R^3} f^\eps\left(1 + |\log f^\eps| + \frac{1}{4}|\xi|^2\right)dx d\xi + \int_{\T^3} \frac{|u^\eps|^2}{2} dx \cr
&\quad + \int_0^T D_2(f^\eps,u^\eps) dt + \frac{1}{\eps}\int_0^T D_1(f^\eps) dt \leq \mathcal{F}(f_0,u_0) + C(T).
\end{aligned}
\end{align}
By expanding the square, one can check after some tedious computations that
\begin{align}\label{hydro-est-1}
\begin{aligned}
&\frac{1}{2}\int_{\T^6 \times \R^6} f^\eps(x,\xi)f^\eps(y,\xi_*)|\xi- \xi_*|^2 dx d\xi dy d\xi_* +\int_{\T^3} \rho_{f^\eps}|u^\eps - u_{f^\eps}|^2 dx \cr
&\quad = \int_{\T^3 \times \R^3} |u^\eps - \xi|^2 f^\eps dx d\xi + \frac{1}{2} \int_{\T^3 \times \T^3} \rho_{f^\eps}(x) \rho_{f^\eps}(y)|u_{f^\eps}(x) - u_{f^\eps}(y)|^2 dx dy.
\end{aligned}
\end{align}
Now we use the similar estimates in \cite[Lemma B.3]{KMT2} to get
\begin{align}\label{hydro-est-2}
\begin{aligned}
&\frac{1}{2}\int_{\T^3 \times \T^3} \rho_{f^\eps}(x) \rho_{f^\eps}(y) | u_{f^\eps}(x) - u_{f^\eps}(y)|^2 dx dy \cr
& \qquad \leq -3(M_0(f_0))^2 + C(T)\eps + \frac{1}{2\eps}D_1(f^\eps) + \int_{\T^6 \times \R^6} f^\eps(x,\xi)f^\eps(y,\xi_*)\frac{|\xi - \xi_*|^2}{2} dx d\xi dy d\xi_*.
\end{aligned}
\end{align}
Combining \eqref{hydro-est-0}, \eqref{hydro-est-1}, and \eqref{hydro-est-2}, we obtain
\begin{align}\label{hydro-better-est-2}
\begin{aligned}
&\mathcal{F}(f^\eps,u^\eps) + \frac{1}{2\eps} \int_0^t D_1(f^\eps)\, ds + \int_0^t \int_{\T^3} \rho_{f^\eps}|u^\eps - u_{f^\eps}|^2 dx ds + \mu\int_0^t \int_{\T^3} |\nabla u^\eps|^2 dx ds \cr
&\qquad \leq \mathcal{F}(f_0,u_0) + C(T) \eps,
\end{aligned}
\end{align}
where we used the fact that $f_0$ has a unit mass, i.e., $M_0(f_0) = 1$.
In light of the above arguments, we conclude the following
proposition.
\begin{proposition}
Suppose the initial data $(f_0, u_0)$ satisfies \eqref{init_cond1}.
Then for any $T>0$ and $\eps >0$ there exists at least one weak solution $(f^\eps, u^\eps)$ to \eqref{sing-V-INS}-\eqref{ini-sing-v-ins} on the time-interval $(0, T)$ satisfying \eqref{hydro-better-est-1} and \eqref{hydro-better-est-2}.
\end{proposition}
We will prove Theorem \ref{lim-thm} through a relative entropy argument. For this to be rigorous, we need a unique strong solution (at least for short time) to the system \eqref{hydro-eqn}-\eqref{ini-hydro-eqn}.
We claim the following result.
\newcommand{\mc}{\mathcal{C}}
\begin{theorem}
Let $s \geq 3$. Suppose the initial data $(\rho_{f_0},u_{f_0},u_0) \in H^s(\T^3)$ and $\rho_{f_0} > 0$. Then there exists a positive constant $T^* > 0$ such that the Cauchy problem \eqref{hydro-eqn}-\eqref{ini-hydro-eqn} has a unique solution $(\rho_f,u_f,u)$ satisfying
\begin{equation*}
    \begin{split}
        (\rho_f,u_f) &\in \mc([0,T^*];H^s)\cap \mc^1([0,T^*];H^{s-1}), \\
        u &\in \mc([0,T^*];H^s) \cap L^2([0,T^*];H^{s+1}).
    \end{split}
\end{equation*}
\end{theorem}
Since local existence theories for this type of balance laws have been well developed, we omit this proof. We refer to \cite{Majda} for the readers who are interested in it.


\subsection{Relative entropy}
We shall prove Theorem \ref{lim-thm} using
a relative entropy argument.
For this purpose, it will be convenient to write the equation
in a more abstract form using the variables
\begin{equation*}
    U := \begin{pmatrix}
        \vr_f \\
        m_f \\
        u
    \end{pmatrix},
    \quad
    A(U) := \begin{pmatrix}
        m_f & 0 & 0 \\
        \frac{m_f \otimes m_f}{\vr_f} & \vr_f & 0 \\
        u \otimes u & 0 & 0
    \end{pmatrix},
\end{equation*}
and
\begin{equation*}
    F(U) := \begin{pmatrix}
        0 \\
        \vr_f u-m_f \\
        m_f - \vr_f u - \Grad p + \mu \Delta u
    \end{pmatrix},
\end{equation*}
where $m_f = \rho_f u_f$.
The system can then be recast in the form
\begin{equation*}
    U_t + \Div A(U) = F(U),
\end{equation*}
and the macroscopic entropy (energy) takes the form
\begin{equation*}
    E(U) := \vr_f \log \vr_f + \frac{|m_f|^2}{2\vr_f} + \frac{|u|^2}{2}.
\end{equation*}

Using the newly defined variables, we define the
relative entropy functional as follows:
\begin{equation}\label{def-re-en}
    \H(V|U) := E(V)- E(U) - dE(U)(V-U), \quad \mbox{and} \quad V := \begin{pmatrix}
        \vr_{\bar{f}} \\
        m_{\bar{f}} \\
        \bar{u}
    \end{pmatrix},
\end{equation}
Upon noticing that
\begin{align*}
\begin{aligned}
&-dE(U)(V-U) = -
\begin{pmatrix}
-\frac{m_f^2}{2\rho_f^2} + \log \rho_f + 1 \\
\frac{m_f}{\rho_f}\\
u
\end{pmatrix}
\begin{pmatrix}
\rho_{\bar{f}} - \rho_f\\
m_{\bar{f}} - m_f\\
\bar{u} - u
\end{pmatrix}
\cr
&\qquad \qquad  =\frac{\rho_{\bar{f}}}{2}|u_f|^2 - \frac{\rho_f}{2}|u_f|^2 - (\log \rho_f + 1)(\rho_{\bar{f}} - \rho_f) + \rho_f |u_f|^2 - \rho_{\bar{f}}u_{\bar{f}} \cdot u_f - (\bar{u} - u)\cdot u,
\end{aligned}
\end{align*}
we see that the relative entropy can alternatively be written
\[
\H(V|U) = \frac{\rho_{\bar{f}}}{2}| u_f - u_{\bar{f}}|^2 + \frac{1}{2}|\bar{u} - u|^2 + P(\rho_{\bar{f}},\rho_f),
\]
where
\begin{align*}
\begin{aligned}
P(\rho_{\bar{f}},\rho_f) := \rho_{\bar{f}} \log \rho_{\bar{f}} - \rho_f \log \rho_f + (\rho_f - \rho_{\bar{f}})(1 + \log \rho_f) \geq \frac{1}{2}\min \left\{ \frac{1}{\rho_{\bar{f}}}, \frac{1}{\rho_f} \right\} (\rho_{\bar{f}} - \rho_f)^2.
\end{aligned}
\end{align*}
Hence, the relative entropy controls the $L^2$-difference
provided one of the densities is without  vacuum regions.

To proceed, we shall need to derive
an evolution equation for the integrated relative entropy.

\begin{lemma}\label{lem-hyd-rel-1}
The relative entropy $\H$ defined in \eqref{def-re-en} satisfies the following equality.
\begin{align*}
\begin{aligned}
&\frac{d}{dt}\int_{\T^3} \H(V|U)\,dx + \mu \int_{\T^3} |\nabla (u - \bar{u})|^2 dx + \int_{\T^3} \rho_{\bar{f}}|(u_{\bar{f}} - \bar{u}) - (u_f - u)|^2 dx\cr
&\quad = \int_{\T^3} \partial_t E(V)\,dx + \int_{\T^3} \rho_{\bar{f}}|\bar{u} - u_{\bar{f}}|^2 dx + \mu \int_{\T^3} |\nabla \bar{u}|^2 dx \cr
&\qquad - \int_{\T^3} \nabla (dE(U)):A(V|U)\,dx - \int_{\T^3} dE(U)\left[ V_t + \nabla \cdot A(V) - F(V)\right]dx\cr
&\qquad - \int_{\T^3} (\rho_f - \rho_{\bar{f}})(\bar{u} - u)(u - u_f)\, dx,
\end{aligned}
\end{align*}
where we have introduced the relative flux functional
\[
A(V|U) := A(V) - A(U) -dA(U)(V-U).
\]
\end{lemma}
\begin{proof} Although this lemma is essential for the proof of Theorem \ref{lim-thm}, it is rather lengthy and technical. Thus we postpone its proof in Appendix \ref{app-a} for the smooth flow of reading.
\end{proof}

\subsection{Relative entropy bound}

The proof of Theorem \ref{lim-thm}
will follow as a consequence of the following
proposition.
\begin{proposition}\label{key-prop} Suppose all assumptions in Theorem \ref{lim-thm}. Set
\[
U := \begin{pmatrix} \rho_f \\ \rho_f u_f \\u \end{pmatrix} \quad \mbox{and} \quad U^\eps := \begin{pmatrix} \rho_{f^\eps} \\ \rho_{f^\eps} u_{f^\eps} \\ u^\eps \end{pmatrix},
\]
where $(\rho_f, u_f,u)$ and $(f^\eps,u^\eps)$ are a unique strong solution to the system \eqref{hydro-eqn}-\eqref{ini-hydro-eqn} and weak solutions to the system \eqref{sing-V-INS}-\eqref{ini-sing-v-ins}, respectively. Then we have
\begin{align*}
\begin{aligned}
&\int_{\T^3} \H(U^\eps|U)(t)\,dx + \mu\int_0^t\int_{\T^3} |\nabla(u - u^\eps)|^2 dx ds + \frac12\int_0^t\int_{\T^3} \rho_{f^\eps}|(u_{f^\eps} - u^\eps) - (u_f - u)|^2 dxds \cr
& \qquad \leq C\sqrt{\eps},
\end{aligned}
\end{align*}
for all $t \in [0,T^*]$.
\end{proposition}
\begin{proof}
From Lemma \ref{lem-hyd-rel-1}, we know that
\begin{align*}
\begin{aligned}
&\int_{\T^3} \H(U^\eps|U)(t)\,dx + \mu \int_0^t\int_{\T^3} |\nabla (u - u^\eps)|^2 dxds + \int_0^t\int_{\T^3} \rho_{f^\eps}|(u_{f^\eps} - u^\eps) - (u_f - u)|^2 dxds\cr
&\,\, = \int_0^t\int_{\T^3} \partial_t E(U^\eps) + \rho_{f^\eps}|u^\eps - u_{f^\eps}|^2 + \mu |\nabla u^\eps|^2 dxds  - \int_0^t\int_{\T^3} \nabla (dE(U)):A(U^\eps|U) dxds \\
&\,\,\, - \int_0^t\int_{\T^3} dE(U)\left[ U^\eps_t + \nabla \cdot A(U^\eps) - F(U^\eps)\right]dxds - \int_0^t\int_{\T^3} (\rho_f - \rho_{f^\eps})(u^\eps - u)(u - u_f)\,dxds,\cr
& \,\, =: \sum_{i=1}^4I_i.
\end{aligned}
\end{align*}
$\bullet$ Estimate of $I_1$: We first notice that
$\int_{\T^3} E(U^\eps)\,dx \leq \mathcal{F}(f^\eps,u^\eps)$, where $\mathcal{F}$ is given in \eqref{hydro:nota}. Then we obtain
\begin{align*}
\begin{aligned}
I_1(t) &= \int_{\T^3} \left( E(U^\eps)(t) - E(U_0) \right) dx +\int_0^t \int_{\T^3} \rho_{f^\eps}|u^\eps - u_{f^\eps}|^2 + \mu |\nabla u^\eps|^2 dxds\cr
&= \int_{\T^3} E(U^\eps)(t)\,dx - \mathcal{F}(f^\eps,u^\eps)(t) \cr
&\quad+ \mathcal{F}(f^\eps,u^\eps)(t) +\int_0^t \int_{\T^3} \rho_{f^\eps}|u^\eps - u_{f^\eps}|^2 + \mu |\nabla u^\eps|^2 dxds - \mathcal{F}(f_0,u_0)\cr
&\quad+  \mathcal{F}(f_0,u_0) - \int_{\T^3} E(U_0)\,dx \cr
&\leq C(T^*)\eps,
\end{aligned}
\end{align*}
where we used the facts that \eqref{hydro-better-est-2} and $\mathcal{F}(f_0,u_0) = \int_{\T^3} E(U_0)\,dx$. \newline
$\bullet$ Estimate of $I_2$: Straightforward computation shows that
\begin{align*}
\begin{aligned}
A(U^\eps|U) &= A(U^\eps) - A(U) - dA(U)(U^\eps-U)\cr
& =
\begin{pmatrix}
0 & 0 & 0 \\
\rho_{f^\eps}(u_{f^\eps} - u_f) \otimes (u_{f^\eps} - u_f) & 0 & 0 \\
(u^\eps - u) \otimes (u^\eps - u) & 0 & 0
\end{pmatrix}.
\end{aligned}
\end{align*}
This implies that
\begin{align*}
\begin{aligned}
\int_{\T^3} |A(U^\eps|U)|\,dx = \int_{\T^3} \rho_{f^\eps}|u_{f^\eps} - u_f|^2 dx + \int_{\T^3} |u^\eps - u|^2 dx \leq 2\int_{\T^3} \H(U^\eps|U)\,dx,
\end{aligned}
\end{align*}
and
\[
I_2(t) \leq C\int_0^t\int_{\T^3} \H(U^\eps|U)\,dxds.
\]
$\bullet$ Estimate of $I_3$: One can find that $\rho_{f^\eps}$ and $u_{f^\eps}$ satisfy
\begin{align*}
\begin{aligned}
& \partial_t\rho_{f^\eps} + \nabla \cdot (\rho_{f^\eps}u_{f^\eps}) = 0,\cr
& \partial_t(\rho_{f^\eps}u_{f^\eps}) + \nabla \cdot (\rho_{f^\eps}u_{f^\eps} \otimes u_{f^\eps}) + \nabla \rho_{f^\eps} - \rho_{f^\eps}(u_{f^\eps} - u^\eps) \cr
&\qquad \qquad = \nabla\cdot\left( \int_{\R^3} \left(u_{f^\eps}\otimes u_{f^\eps} - \xi \otimes \xi + \mathbb{I}\right)f^\eps d\xi\right),
\end{aligned}
\end{align*}
in the distribution sense on $\T^3 \times [0,T^*)$. Then we deduce that
\begin{align*}
\begin{aligned}
&\left| \int_0^t\int_{\T^3} dE(U)[U_s^\eps + \nabla \cdot A(U^\eps) - F(U^\eps)]dxdt\right|\cr
&\quad \leq \left| \int_0^t\int_{\T^3}|\nabla dE(U)|\left|\int_{\R^3} \left(u_{f^\eps}\otimes u_{f^\eps} - \xi \otimes \xi + \mathbb{I}\right)f^\eps d\xi \right|dxdt\right|\cr
&\quad \leq C\int_0^t\int_{\T^3}\left|\int_{\R^3} \left(u_{f^\eps}\otimes u_{f^\eps} - \xi \otimes \xi + \mathbb{I}\right)f^\eps d\xi \right|dxdt.
\end{aligned}
\end{align*}
Then we now apply the same argument in \cite[Lemma 4.8]{KMT2} to have
\[
\left|\int_0^t \int_{\T^3} dE(U)\left[ U_s^{\eps} + \nabla \cdot A(U^\eps) - F(U^\eps) \right] dx ds \right| \leq \sqrt{\eps}C(T^*).
\]
$\bullet$ Estimate of $I_4$: The strategy is to use the third term in the dissipation. By Cauchy-Schwartz inequality and using the fact
\[
1 \leq \min\left( \frac1x, \frac1y\right)(x+y), \quad \mbox{for} \quad x,y >0,
\]
we get
\begin{align}\label{estI4}
\begin{aligned}
&\left| \int_{\T^3} (\rho_{f} - \rho_{f^\eps})(u^\eps - u)(u - u_f) dx \right| \cr
& \quad \leq \|u-u_f\|_{L^\infty}\left( \int_{\T^3} \min\left( \frac{1}{\rho_f}, \frac{1}{\rho_{f^\eps}} \right)|\rho_f - \rho_{ f^\eps}|^2 dx\right)^\frac12\left( \int_{\T^3} (\rho_f + \rho_{f^\eps})|u - u^\eps|^2 dx\right)^\frac12\cr
&\quad \leq C\int_{\T^3} \H(U^\eps|U)\,dx + \frac14\int_{\T^3} \rho_{f^\eps}|u - u^\eps|^2 dx.
\end{aligned}
\end{align}
On the other hand, the second term of the last inequality in \eqref{estI4} is again estimated as follows.
\begin{align*}
\begin{aligned}
\int_{T^3} \rho_{f^\eps}| u - u^\eps|^2 dx &= \int_{\T^3} \rho_{f^\eps} | u - u_f + u_f - u_{f^\eps} + u_{f^\eps} - u^\eps|^2 dx\cr
&\leq 2\int_{\T^3} \rho_{f^\eps} |(u - u_f) - (u^\eps - u_{f^\eps})|^2 dx + 2\int_{\T^3} \rho_{f^\eps}|u_f - u_{f^\eps}|^2 dx, \\
\end{aligned}
\end{align*}
and this implies
\begin{align*}
\begin{aligned}
&\frac14\int_{\T^3} \rho_{f^\eps}|u - u^\eps|^2 dx \\
&\qquad\leq \frac12\int_{\T^3} \rho_{f^\eps}|u_f - u_{f^\eps}|^2 dx + \frac12\int_{\T^3} \rho_{f^\eps} |(u - u_f) - (u^\eps - u_{f^\eps})|^2 dx\cr
&\qquad\leq \int_{\T^3} \H(U^\eps|U)\,dx + \frac12\int_{\T^3} \rho_{f^\eps} |(u - u_f) - (u^\eps - u_{f^\eps})|^2 dx.
\end{aligned}
\end{align*}
This concludes that
\begin{equation*}
    \begin{split}
        &\left| \int_{\T^3} (\rho_{f} - \rho_{f^\eps})(u^\eps - u)(u - u_f) dx \right|\\
        &\qquad \leq C\int_{\T^3} \H(U^\eps|U)\,dx + \frac12\int_{\T^3} \rho_{f^\eps} |(u - u_f) - (u^\eps - u_{f^\eps})|^2 dx.
    \end{split}
\end{equation*}
From the above, we have
\begin{align*}
\begin{aligned}
&\int_{\T^3} \H(U^\eps|U)(t)\,dx + \mu \int_0^t\int_{\T^3} |\nabla (u - u^\eps)|^2 dxds \\
&\qquad \quad+ \frac12\int_0^t\int_{\T^3} \rho_{f^\eps}|(u_{f^\eps} - u^\eps) - (u_f - u)|^2 dxds\cr
&\qquad \leq C\sqrt{\eps} + C\int_0^t\int_{\T^3} \H(U^\eps|U)(s)\,dxds, \quad \mbox{for all} \quad t \in [0,T^*].
\end{aligned}
\end{align*}
We now apply the Gronwall's inequality to derive that
\begin{align*}
\begin{aligned}
&\int_{\T^3} \H(U^\eps|U)(t)\,dx + \mu\int_0^t\int_{\T^3} |\nabla(u - u^\eps)|^2 dx ds \\
&\qquad \quad+ \frac12\int_0^t\int_{\T^3} \rho_{f^\eps}|(u_{f^\eps} - u^\eps) - (u_f - u)|^2 dxds  \leq C\sqrt{\eps}.
\end{aligned}
\end{align*}
\end{proof}
\subsection{Proof of Theorem \ref{lim-thm}}
 The entropy inequality in Proposition \ref{key-prop} and arguments in \cite{KMT2} yield that
\begin{align*}
\begin{aligned}
&\rho_{f^\eps} \to \rho_f\quad \mbox{in } L^1_{loc}([0,T^*];L^1(\T^3)),\cr
&\rho_{f^\eps}u_{f^\eps} \to \rho_f u_f \quad \mbox{in } L^1_{loc}([0,T^*];L^1(\T^3)),\cr
&\rho_{f^\eps}|u_{f^\eps}|^2 \to \rho_f |u_f|^2 \quad \mbox{in } L^1_{loc}([0,T^*];L^1(\T^3)),\cr
&u^\eps \to u \quad \mbox{in } L^1_{loc}([0,T^*];L^2(\T^3)).
\end{aligned}
\end{align*}
Furthermore, we can also use the same argument in \cite{KMT2} to conclude that
\[
f^\eps \to \frac{\rho_f}{(2\pi)^\frac32} e^{-\frac{|\xi - u_f|^2}{2}} \quad \mbox{in } L^1_{loc}([0,T^*];L^1(\T^3 \times \R^3)).
\]
This completes the proof.
\qed

%
%
%
%
\section{A priori estimate of asymptotic behavior (Theorem \ref{a-priori-est-lt})}
In this section, we provide a long-time behavior estimate for the system \eqref{V-INS}-\eqref{ini-V-INS} without diffusion, i.e., $\sigma = 0$. Since the constants $\alpha$ and $\beta$ do not play any crucial role in our analysis as we mentioned before, we assume that $\alpha = \beta = 1$. For the estimate of large-time behaviour, we first notice that local density $\rho_f$ and velocity $u_f$ in \eqref{def-locals} satisfy the following hydrodynamic equations.
\begin{align}\label{hydro-lev}
\begin{aligned}
& \partial_t \rho_f + \nabla \cdot (\rho_f u_f) = 0, \cr
& \partial_t(\rho_f u_f) + \nabla \cdot (\rho_f u_f \otimes u_f) + \nabla \cdot \tilde{P} = \rho_f(u - u_f),
\end{aligned}
\end{align}
where $\tilde{P}$ is given by
\[
\tilde{P} := \int_{\bbr^3} (\xi - u_f) \otimes (\xi - u_f) f d\xi.
\]
We recall energy-fluctuation functions ${\mathcal E}_P, \mathcal{E}_U, {\mathcal E}_F$ and  ${\mathcal E}_I$:
\begin{align*}
\begin{aligned}
    {\mathcal E}_P(t) &:= \frac{1}{2}\int_{\bbt^3 \times \bbr^3}|\xi - u_f|^2 f dx d\xi, \\
    \mathcal{E}_U(t) &:= \frac{1}{2}\int_{\bbt^3 \times\bbt^3} |u_f(x) - u_f(y)|^2 \rho_f(x) \rho_f(y) dx dy, \\
    \mathcal{E}_F(t) &:= \frac{1}{2}\int_{\bbt^3} |u - u_c(t) |^{2} dx, \\
    \mathcal{E}_I(t) &:= \frac{1}{2}|u_c(t)-\xi_c(t)|^2,
    \end{aligned}
\end{align*}
Then we next investigate the time-evolution of the above energy-fluctuation functions.
\begin{lemma}\label{each-energy}
Let $(f,u)$ be classical solutions to the system \eqref{V-INS}-\eqref{ini-V-INS} with $\sigma = 0$ satisfying
\[
\lim_{|\xi| \to \infty} |\xi|^2 f (x,\xi,t) = 0, \qquad (x,t) \in \bbt^3 \times [0,T].
\]
The following identities hold:
\begin{align*}
\begin{aligned}
(i)~\frac{d\mathcal{E}_P}{dt} &= \int_{\bbt^3} (\nabla \cdot \tilde{P})\cdot u_f \,dx - 4\mathcal{E}_P. \cr
(ii)~\frac{d\mathcal{E}_U}{dt} &= - 2\int_{\bbt^3} (\nabla \cdot \tilde{P})\cdot u_f \,dx + 2\int_{\bbt^3} \rho_f(u-u_f) \cdot u_f \,dx\\
&\qquad  -2\int_{\bbt^3} \rho_f u_f \,dx \cdot \int_{\bbt^3} \rho_f (u - u_f)\, dx.\cr
(iii)~\frac{d\mathcal{E}_F}{dt} &= -\mu \int_{\bbt^3} |\nabla u |^2 \,dx + \int_{\bbt^3\times \bbr^{3}} (u_c
- u) \cdot (u-\xi) f \,d\xi dx. \cr
(iv)~\frac{d\mathcal
E_I}{dt} &= -2 \int_{\bbt^3 \times \bbr^{3}} (u_c-\xi_c)\cdot
(u-\xi) f \,d\xi dx.
\end{aligned}
\end{align*}
\end{lemma}
\begin{proof}
For the estimate of $(i)$, it follows from the system \eqref{V-INS} that
\begin{align}\label{e-p-est-1}
\begin{aligned}
\frac{d\mathcal{E}_P}{dt} &= -\int_{\bbt^3 \times \bbr^3} (\xi - u_f) \cdot u_f^{\prime} f \,dx d\xi + \frac{1}{2} \int_{\bbt^3 \times \bbr^3} |\xi - u_f|^2  \partial_t f \,dx d\xi \cr
&= \frac{1}{2} \int_{\bbt^3 \times \bbr^3} |\xi - u_f|^2 \Big( -\xi \cdot \nabla f + \nabla_{\xi} \cdot [(\xi - u_f)f] - \nabla_{\xi} \cdot [(u-\xi)f] \Big) dx d\xi\cr
&=: \sum_{i=1}^{3}\mathcal{I}_i,
\end{aligned}
\end{align}
where $u_f^{\prime}:= \frac{d}{dt} u_f$, and $\mathcal{I}_i,i=1,2,3$ are given by
\begin{align}\label{e-p-est-2}
\begin{aligned}
\mathcal{I}_1 &= \frac{1}{2} \int_{\bbt^3 \times \bbr^3} \nabla(|\xi - u_f|^2) \cdot \xi f \,dx d\xi = -\int_{\bbt^3 \times \bbr^3} \big( (\xi - u_f) \cdot \nabla u_f \big) \cdot \xi f \,d\xi dx,\cr
\mathcal{I}_2 &= - \int_{\bbt^3 \times \bbr^3} |\xi - u_f|^2 f \,dx d\xi,\cr
\mathcal{I}_3 &= \int_{\bbt^3 \times \bbr^3} (\xi - u_f) \cdot (u - \xi) f \,dx d\xi = -\int_{\bbt^3 \times \bbr^3} \xi \cdot (\xi - u_f) f \,dx d\xi \cr
&= - \int_{\bbt^3 \times \bbr^3} |\xi - u_f|^2 f \,dx d\xi.
\end{aligned}
\end{align}
A further integration by parts leads to
\begin{align*}
\begin{aligned}
\mathcal{I}_1 &= -\sum_{i,j=1}^{3} \int_{\bbt^3 \times \bbr^3} (\xi^i - u_f^i)\partial_{i} u_f^j\xi^jf \,dx d\xi \\
&=-\sum_{i,j=1}^{3} \int_{\bbt^3 \times \bbr^3} (\xi^i - u_f^i)(\partial_{i} u_f^j) (\xi^j - u_f^j) f \,dx d\xi
= \int_{\bbt^3} (\nabla \cdot \tilde{P})\cdot u_f \,dx.
\end{aligned}
\end{align*}
Thus, (i) is obtained by combining \eqref{e-p-est-1} and \eqref{e-p-est-2} with the above equality.
For the identity (ii), we use the hydrodynamic equations \eqref{hydro-lev} to find
\begin{align*}
\begin{aligned}
\frac{d\mathcal{E}_U}{dt} &=\int_{\bbt^3 \times \bbt^3}\left( u_f(x) - u_f(y)\right) \cdot \left( u_f^{\prime}(x) - u_f^{\prime}(y)\right) \rho_f(x) \rho_f(y) \,dx dy \cr
&\quad + \int_{\bbt^3 \times \bbt^3} |u_f(x) - u_f(y)|^2 \rho_f^{\prime}(x) \rho_f(y) \,dx dy \\
&= 2\int_{\bbt^3} u_f \cdot u_f^{\prime}\rho_f \,dx - 2\left(\int_{\bbt^3} \rho_f u_f \,dx\right) \cdot \left(\int_{\bbt^3} \rho_f u_f^{\prime} \,dx\right) \cr
&\quad + 2\int_{\bbt^3} \rho_f u_f \cdot \left(u_f \cdot \nabla u_f\right) \,dx - 2\left(\int_{\bbt^3} \rho_f \left(u_f \cdot \nabla u_f\right) \,dx\right) \cdot \left(\int_{\bbt^3} \rho_f u_f \,dx \right)\cr
&=- 2\int_{\bbt^3} (\nabla \cdot \tilde{P})\cdot u_f \,dx + 2\int_{\bbt^3} \rho_f(u-u_f) \cdot u_f \,dx \\
&\quad-2\left(\int_{\bbt^3} \rho_f u_f  \,dx\right) \cdot \left(\int_{\bbt^3} \rho_f (u - u_f) \,dx\right),
\end{aligned}
\end{align*}
where we used the fact that $\|\rho_f\|_{L^1(\bbt^3)} = 1$ and
\[
\rho_f u_f^{\prime} + \rho_f u_f \cdot \nabla u_f + \nabla \cdot \tilde{P} = \rho_f (u-u_f).
\]
For the estimate of (iii), we use the definition of ${\mathcal E}_F$ and direct integration by parts to get
\begin{align*}
\begin{aligned}
\frac{d\mathcal{E}_F}{dt} &=\int_{\bbt^3}(u-u_c)\cdot\partial_t
u \,dx - u^{\prime}_c \cdot \int_{\bbt^3}(u-u_c) \,dx  = \int_{\bbt^3}(u-u_c) \cdot \partial_t u \,dx \cr
& = -\int_{\bbt^3}(u\cdot \nabla u)\cdot(u-u_c)\,dx
- \int_{\bbt^3}(u-u_c)\cdot\nabla p \,dx\\
& \quad + \mu \int_{\bbt^3}(u- u_c)\cdot\Delta u \,dx
 -\int_{\bbt^3 \times \bbr^{3}}(u- u_c)\cdot(u-\xi)f \,dx d\xi\\
&=-\mu \int_{\bbt^3}|\nabla u|^2 \,dx-\int_{\bbt^3 \times \bbr^{3}}(u- u_c)\cdot(u-\xi)f \,dx d\xi,
\end{aligned}
\end{align*}
since $\nabla \cdot u = 0$. Finally we employ the following facts
\[ \xi^{\prime}_c = \int_{\bbt^3 \times
\bbr^3} \big(u - \xi \big) f \,d\xi dx    \quad \mbox{and} \quad u^{\prime}_c =-\int_{\bbt^3 \times
\bbr^3} \big(u - \xi \big) f \,d\xi dx, \]
to derive the estimate of (iv)
\[
\frac{d{\mathcal E}_I}{dt} = (u_c-\xi_c)\cdot(u^{\prime}_c - \xi^{\prime}_c) = -2(u_c -\xi_c) \cdot \int_{\bbt^3 \times \bbr^{3}} (u-\xi) f d\xi dx.
\]
\end{proof}

\begin{remark}
Since $\int_{\T^3}\vr_f \,dx \equiv 1$, we have that
\[
\int_{\bbt^3} (u_f - \xi_c)\rho_f \,dx = \int_{\bbt^3 \times \bbr^3} (\xi - \xi_c)f \,d\xi dx = 0.
\]
As a consequence,
\begin{align*}
\begin{aligned}
\mathcal{E}_U(t) &= \frac{1}{2}\int_{\bbt^3 \times\bbt^3} |u_f(x) - u_f(y)|^2 \rho_f(x) \rho_f(y) \,dx dy \cr
&=\frac{1}{2}\int_{\bbt^3 \times\bbt^3} |u_f(x) - \xi_c + \xi_c - u_f(y)|^2 \rho_f(x) \rho_f(y) \,dx dy \cr
&=\frac{1}{2}\int_{\bbt^3 \times \bbt^3} |u_f(x) - \xi_c|^2 \rho_f(x)\rho_f(y) \,dx dy \\
&\quad+ \frac{1}{2}\int_{\bbt^3 \times \bbt^3} |u_f(y) - \xi_c|^2 \rho_f(x)\rho_f(y) \,dx dy  - \left( \int_{\bbt^3} (u_f - \xi_c)\rho_f \,dx\right)^2\cr
&= \int_{\bbt^3} |u_f - \xi_c|^2 \rho_f \,dx.
\end{aligned}
\end{align*}
\end{remark}

\begin{proof}[Proof of Theorem \ref{a-priori-est-lt}]
For the sake of the reader, we divide the proof into two steps. \newline

1. In this part, we will show that
\begin{align}\label{local-e-est-1}
\begin{aligned}
(i) &\frac{d}{dt} \Big( 2\mathcal{E}_P + \mathcal{E}_U\Big) = -8\mathcal{E}_P - 2\mathcal{E}_U + 2\int_{\bbt^3 \times \bbr^3} (\xi - \xi_c)\cdot u f d\xi dx. \cr
(ii) & \frac{d}{dt} \Big( 2\mathcal{E}_F + \mathcal{E}_I\Big) = 2\mathcal{E}_U + 4\mathcal{E}_P -2\mu \int_{\bbt^3} |\nabla u |^2 dx - 2\int_{\bbt^3 \times \bbr^3} |u - \xi|^2 f dx d\xi \\
&\qquad \qquad \qquad \quad- 2\int_{\bbt^3 \times \bbr^3} (\xi - \xi_c)\cdot u f d\xi dx.
\end{aligned}
\end{align}
For a detailed estimate of (i), we use Lemma \ref{each-energy} to get
\begin{align*}
\begin{aligned}
\frac{d}{dt} \Big( 2\mathcal{E}_P + \mathcal{E}_U\Big)
&= -8\mathcal{E}_P + 2\int_{\bbt^3} \rho_f (u- u_f) \cdot (u_f - \xi_c) \,dx \cr
&= -8\mathcal{E}_P + 2\int_{\bbt^3} \rho_f (u - \xi_c) \cdot (u_f - \xi_c) \,dx - 2\int_{\bbt^3} \rho_f |u_f - \xi_c|^2 dx \cr
&= -8\mathcal{E}_P - 2\mathcal{E}_U + 2\int_{\bbt^3} \rho_f u \cdot(u_f - \xi_c) \,dx,
\end{aligned}
\end{align*}
where we used
\[
\int_{\bbt^3} \rho_f u_f \,dx = \xi_c \quad \mbox{and} \quad \int_{\bbt^3} \rho_f (u_f - \xi_c) \,dx = 0.
\]
For the second part (ii), it also follows from the Lemma \ref{each-energy} that
\begin{align}\label{local-e-est-2}
\begin{aligned}
\frac{d}{dt} \Big( 2\mathcal{E}_F + \mathcal{E}_I\Big)  &=  -2\mu \int_{\bbt^3} |\nabla u |^2 dx - 2\int_{\bbt^3 \times \bbr^3} (u - \xi) \cdot (u- \xi_c) f dx d\xi \cr
&=  -2\mu \int_{\bbt^3} |\nabla u |^2 dx - 2\int_{\bbt^3 \times \bbr^3} |u - \xi|^2 f dx d\xi \\
&\quad - 2\int_{\bbt^3 \times \bbr^3} u \cdot (\xi - \xi_c) f dx d\xi + 2\int_{\bbt^3 \times \bbr^3} \xi \cdot (\xi - \xi_c) f dx d\xi.
\end{aligned}
\end{align}
On the other hand, the fourth term in the last inequality of \eqref{local-e-est-2} is estimated by
\[
2\int_{\bbt^3 \times \bbr^3} \xi \cdot (\xi - \xi_c)f dx d\xi = 2\int_{\bbt^3 \times \bbr^3} |\xi - \xi_c|^2f dx d\xi = 2\mathcal{E}_U + 4\mathcal{E}_P.
\]
This yields the estimate of (ii). \newline
2. We now combine two inequalities in \eqref{local-e-est-1} to find
\[
\frac{d}{dt} \mathcal{E}(t) =  -4\mathcal{E}_P - 2\mu \int_{\bbt^3} |\nabla u |^2 dx - 2\int_{\bbt^3 \times \bbr^3} |u - \xi|^2 f dx d\xi.
\]
We set a corresponding dissipation function $\mathcal{D}(t)$ to $\mathcal{E}(t)$:
\[
\mathcal{D}(t) := 4\mathcal{E}_P + 2\mu \int_{\bbt^3} |\nabla u |^2 dx + 2\int_{\bbt^3 \times \bbr^3} |u - \xi|^2 f dx d\xi.
\]
Then we obtain
\[
\frac{d}{dt}\mathcal{E}(t) + \mathcal{D}(t) = 0.
\] 

{\it Claim:} there exists a positive constant $C>0$ such that $\mathcal{E}(t) \leq C\mathcal{D}(t)$ for all $t \geq 0$. 

For the proof of claim, we estimate the last term in the function $\mathcal{D}(t)$ as follows.
\begin{align}\label{rev-eqn1}
\begin{aligned}
\int_{\bbt^3\times \bbr^3} |u-\xi |^2 f d\xi dx  &= \int_{\bbt^3 \times \bbr^3} |u-u_c + u_c-\xi_c +\xi_c -\xi|^2 f d\xi dx\cr
&= \int_{\bbt^3} \rho_f |u- u_c|^2 dx + |u_c- \xi_c|^2 + \int_{\bbt^3 \times \bbr^3}|\xi_c - \xi |^2 f d\xi dx \cr
&\quad +2\int_{\bbt^3 \times \bbr^3} (u -u_c) \cdot (u_c -\xi) f d\xi dx \cr
&= 2\mathcal{E}_I + 2\mathcal{E}_P + \mathcal{E}_U + \int_{\bbt^3} \rho_f |u- u_c|^2 dx \\
&\quad + 2\int_{\bbt^3 \times \bbr^{3}} (u -u_c) \cdot (u_c -\xi) f d\xi dx,
\end{aligned}
\end{align}
where we used
\[
\int_{\bbt^3 \times \bbr^3} (u_c - \xi_c) \cdot (\xi_c - \xi)f d\xi dx = 0,
\]
and
\begin{equation*}
    \begin{split}
\int_{\bbt^3 \times \bbr^3}|\xi_c - \xi |^2 f d\xi dx &= \int_{\bbt^3 \times \bbr^3}|\xi_c - u_f |^2 f d\xi dx + \int_{\bbt^3 \times \bbr^3}|u_f - \xi |^2 f d\xi dx \\
&= \mathcal{E}_U + 2\mathcal{E}_P.
\end{split}
\end{equation*}
Furthermore we use the fact that 
\begin{align*}
\begin{aligned}
\int_{\bbt^3 \times \bbr^3} |u_c - \xi|^2 f dx d\xi  &= \int_{\bbt^3 \times \bbr^3} |u_c - \xi_c + \xi_c - \xi|^2 f dx d\xi\cr
&= \int_{\bbt^3 \times \bbr^3} (|u_c - \xi_c|^2 + |\xi_c - \xi|^2) f dx d\xi \cr
&= 2\mathcal{E}_I + 2\mathcal{E}_P + \mathcal{E}_U.
\end{aligned}
\end{align*}
to find
\begin{align}\label{rev-eqn2}
\begin{aligned}
- 4\int_{\bbt^3 \times \bbr^3} ( u - u_c) \cdot (u_c - \xi)f dx d\xi &\leq 4 \int_{\bbt^3} \rho_f |u-u_c|^2 dx + \int_{\bbt^3 \times \bbr^3} |u_c - \xi|^2 f dx d\xi \cr
& = 4 \int_{\bbt^3} \rho_f |u-u_c|^2 dx + 2\mathcal{E}_I + 2\mathcal{E}_P + \mathcal{E}_U.
\end{aligned}
\end{align}
Then it follows from \eqref{rev-eqn1} and \eqref{rev-eqn2} that
\[
2\mathcal{E}_I + 2\mathcal{E}_P + \mathcal{E}_U \leq 2\int_{\T^3 \times \R^3} | u - \xi|^2 f dx d\xi + 2\int_{\T^3} \rho_f | u - u_c|^2 dx.
\]
This deduces that
\begin{align*}
\begin{aligned}
\mathcal{E}(t) &\leq \int_{\T^3} | u - u_c|^2 dx + 2\int_{\T^3 \times \R^3} | u - \xi|^2 f dx d\xi + 2\int_{\T^3} \rho_f | u - u_c|^2 dx \cr
&\leq C\left(1 + \|\rho_f\|_{L^\infty(0,\infty;L^{3/2})}\right)\int_{\T^3} |\nabla u|^2 dx + 2\int_{\T^3 \times \R^3} | u - \xi|^2 f dx d\xi\cr
&\leq C\mathcal{D}(t),
\end{aligned}
\end{align*}
where we used the following Sobolev inequalities.
\begin{align*}
\begin{aligned}
\int_{\bbt^3} |u-u_c|^2 dx &\leq C\int_{\bbt^3} |\nabla u|^2 dx, \cr
\int_{\bbt^3}\rho_f |u - u_c|^2 dx &\leq \|\rho_f\|_{L^{3/2}}\|u - u_c\|^2_{L^6} \\
&\leq C\|\rho_f\|_{L^{3/2}}\|u - u_c\|^2_{H^1} \leq C\|\rho_f\|_{L^\infty(0,\infty;L^{3/2})}\|\nabla u\|^2_{L^2}.
\end{aligned}
\end{align*}
This yields the proof of claim, and we have
\[
\frac{d}{dt}\mathcal{E}(t) + C\mathcal{E}(t) \leq 0, \quad t \geq 0,
\]
for some positive constant $C > 0$. This completes the proof.
\end{proof}
\appendix
%
%
%
%
\section{Proof of Lemma \ref{lem-hyd-rel-1}}\label{app-a}

In this part, we provide the proof of Lemma \ref{lem-hyd-rel-1}. It follows from \eqref{def-re-en} that
\begin{align*}
\begin{aligned}
\frac{d}{dt}\int_{\T^3} \H(V|U)\,dx & = \int_{\T^3} \partial_t E(V)dx - \int_{\T^3} dE(U)(V_t + \nabla \cdot A(V)- F(V))\,dx \cr
&\quad +\int_{\T^3}d^2E(U) \nabla \cdot A(U)(V-U) + dE(U) \nabla \cdot A(V)\,dx\cr
&\quad -\int_{\T^3} d^2E(U)F(U)(V-U) + dE(U)F(V)\,dx\cr
&=: \sum_{i=1}^4 I_i.
\end{aligned}
\end{align*}
Using integration by parts, we find
\begin{align*}
\begin{aligned}
I_3 &= \int_{\T^3} \left(\nabla dE(U)\right):\left( dA(U)(V-U) - A(V)\right)dx  \cr
&= -\int_{\T^3} \left(\nabla dE(U)\right):\left(A(V|U) + A(U)\right)dx\cr
&=-\int_{\T^3} \left(\nabla dE(U)\right):A(V|U)\,dx.
\end{aligned}
\end{align*}
Here we used the fact that
\[
\int_{\T^3} \left(\nabla dE(U)\right): A(U)\,dx = \int_{\T^3} \nabla \cdot Q(U)\,dx = 0,
\]
where $Q$ is an entropy flux function given by
\[
Q_i(U) := \sum_{k}A_{ki}(U)d_kE(U).
\]
For the estimate $I_4$, we claim that the following identity holds.
\begin{equation}\label{eq:longg}
    \begin{split}
        &\int_{\T^3}d^2E(U)F(U)(V-U) + dE(U)F(V)\,dx \\
        &\qquad = -\int_{\T^3}\vr_{\bar f}\left|u_{\bar f} - \bar u\right|^2 \,dx - \mu\int_{\T^3}|\Grad \bar u|^2\,dx \\
        &\qquad \quad + \int_{\T^3}\vr_{\bar f}\left|(u_f- u) - (u_{\bar f} - \bar u)\right|^2\,dx
        + \mu \int_{\T^3}\left|\Grad (u-\bar u) \right|^2\,dx\\
        &\qquad \quad+ \int_{\T^3}(\vr_f - \vr_{\bar f})(\bar u - u)(u_f - u)\,dx.
    \end{split}
\end{equation}
{\it Proof of claim:} We first notice that
\begin{equation*}
    dE(U) = \begin{pmatrix}
        \log \vr_f + 1 -\frac{m_f^2}{2\vr_f^2} \\
        \frac{m_f}{\vr_f} \\
        u
    \end{pmatrix}
    \quad \mbox{and} \quad
    d^2E(U) = \begin{pmatrix}
        * & - \frac{m_f}{\vr_f^2} & 0 \\
        * & \frac{1}{\vr_f} & 0 \\
        0 & 0 & 1
    \end{pmatrix}.
\end{equation*}
Then by direct calculation, we have
\begin{equation}\label{long:1}
    \begin{split}
        &\int_{\T^3}d^2E(U)F(U)(V-U)dx \\
        &\qquad=\int_{\T^3} \vr_f\left[u_f \bar u - u_f u - u \bar u + u^2\right]
                    + \vr_{\bar f}\left[-u_f u + u_f^2 + u_{\bar{f}}u - u_{\bar f}u_f\right]~dx \\
        &\qquad \quad +\int_{\T^3} (\bar u - u)\mu \Delta u - (\bar u - u)\Grad p~dx.
    \end{split}
\end{equation}
and moreover
\begin{equation}\label{long:2}
    \begin{split}
        \int_{\T^3}dE(U)F(V)~dx
        &= \int_{\T^3} \vr_{\bar f}\left[u_f \bar u - u_{\bar f}u_f + u_{\bar f}u - \bar u u \right]~dx \\
        &\quad + \int_{\T^3}  \mu \,u \Delta \bar u - u \Grad \bar p~dx.
    \end{split}
\end{equation}
By combining \eqref{long:1}-\eqref{long:2}, and using that $\Div u = \Div \bar u = 0$, we obtain
\begin{equation}\label{long:collect}
    \begin{split}
      &\int_{\T^3}d^2E(U)F(U)(V-U) + dE(U)F(V)~dx \\
      &\qquad\quad = \int_{\T^3} \vr_{\bar f}\left[-u_f u + u_f^2 - 2u_{\bar f}u_f + 2 u_{\bar f}u+ u_f \bar u - \bar u u\right]~dx \\
      &\qquad\qquad +\int_{\T^3} \vr_f \left[u_f \bar u - u_f u - u \bar u + u^2\right]~ dx \\
      &\qquad\qquad + \int_{\T^3} \mu \left[u \Delta \bar u + \bar u \Delta u - u \Delta u\right]~dx
      =: J_1 + J_2 + J_3.
    \end{split}
\end{equation}
By adding and subtracting, we rewrite $J_1$ as follows.
\begin{equation*}
    \begin{split}
         J_1=& \int_{\T^3} \vr_{\bar f}\left[-u_f u + u_f^2 - 2u_{\bar f}u_f + 2 u_{\bar f}u+ u_f \bar u - \bar u u\right]~dx \\
        =& \int_{\T^3}\vr_{\bar f}\left[-2(u_{\bar f}- \bar u)(u_f - u)- \bar u(u_f - u) - u_f u + u_f^2\right]~dx.
    \end{split}
\end{equation*}
Next, we add and subtract $\vr_{\bar f}|u_{f} - u|^2$ to discover
\begin{equation*}
    \begin{split}
         J_1 &= \int_{\T^3} \vr_{\bar f}\left[-2(u_{\bar f}- \bar u)(u_f - u) + |u_{f} -  u|^2\right]dx \\
        &\quad + \int_{\T^3} \vr_{\bar f}\left[-|u_{f} -  u|^2- \bar u(u_f - u) - u_f u + u_f^2\right]dx \\
        &= \int_{\T^3} \vr_{\bar f} \left|(u_{\bar f} - \bar u) - (u_f - u)\right|^2~dx - \int_{\T^3}\vr_{\bar f}|u_{\bar f} - \bar u|^2dx \\
        &\quad + \int_{\T^3}\vr_{\bar f}\left[u_f u - u^2 - \bar uu_f + \bar u u\right]dx.
    \end{split}
\end{equation*}
As a consequence, we find that
\begin{equation}\label{long:II}
    \begin{split}
         J_1+ J_2 &= \int_{\T^3} \vr_{\bar f} \left|(u_{\bar f} - \bar u) - (u_f - u)\right|^2dx - \int_{\T^3}\vr_{\bar f}|u_{\bar f} - \bar u|^2dx \\
        &\quad + \int_{\T^3}(\vr_{\bar f}-\vr_f)\left[u_f u - u^2 - \bar uu_f + \bar u u\right]dx\\
        &= \int_{\T^3} \vr_{\bar f} \left|(u_{\bar f} - \bar u) - (u_f - u)\right|^2dx - \int_{\T^3}\vr_{\bar f}|u_{\bar f} - \bar u|^2dx \\
        &\quad + \int_{\T^3}(\vr_{\bar f}-\vr_f)(u - \bar u)(u_f - u)\,dx.
    \end{split}
\end{equation}
Next, we apply integration by parts to write $J_3$ in the form
\begin{equation}\label{long:III}
    \begin{split}
        J_3 = \int_{\T^3} \mu \left[u \Delta \bar u + \bar u \Delta u - u \Delta u\right]dx = -\mu\int_{\T^3}|\Grad \bar u|^2dx+ \mu\int_{\T^3}|\Grad (u-\bar u)|^2dx.
    \end{split}
\end{equation}
By setting \eqref{long:II} and \eqref{long:III} in \eqref{long:collect},
we obtain \eqref{eq:longg}.
\section*{Acknowledgments}
JAC was partially supported by the project MTM2011-27739-C04-02
DGI (Spain) and 2009-SGR-345 from AGAUR-Generalitat de Catalunya.
JAC acknowledges support from the Royal Society by a Wolfson
Research Merit Award. YPC was supported by Basic Science Research
Program through the National Research Foundation of Korea funded
by the Ministry of Education, Science and Technology (ref.
2012R1A6A3A03039496). JAC and YPC were supported by Engineering
and Physical Sciences Research Council grants with references
EP/K008404/1 (individual grant) and EP/I019111/1 (platform grant).
The work of TK was supported by the Norwegian Research Council (proj. 205738).

%
%
%
%

\end{document}